\documentclass[11pt,a4paper,oneside]{article}
\usepackage[margin=1in]{geometry}
\usepackage[usenames,dvipsnames]{xcolor} %for colored links
\usepackage[pdftitle={},
pdfauthor={Jacopo Borga},
colorlinks=true,linkcolor=JungleGreen,urlcolor=OliveGreen,citecolor=PineGreen,bookmarks=true,bookmarksopen=true,bookmarksopenlevel=2,unicode=true,linktocpage]{hyperref}

\usepackage[english]{babel}%for the language
\usepackage{authblk}  %for affiliation author
\usepackage{amsmath}  %for several maths commands
\usepackage{graphicx}  % for figures, here an example:
\usepackage{amsthm} % for different theorems styles		
\usepackage[capitalize]{cleveref} %for fast citations (Keep it after \usepackage{amsthm})		
\usepackage{autonum}		
\usepackage{amsmath,bm}  % For bold notation
\usepackage{amsfonts} %for several math tools
\usepackage{mathtools} %for several math tools like\coloneqq
\usepackage{dsfont}% for fancy sign like \mathds{1}
\usepackage{bigints}% for big integrals
\usepackage{stackengine}% for definining the command to "underbar"
\usepackage{multicol} %for multicolumn itemize

\newtheorem{thm}{Theorem}[section]

\newtheorem{prop}[thm]{Proposition}
\newtheorem{lem}[thm]{Lemma}
\newtheorem{conj}[thm]{Conjecture}

\theoremstyle{definition}

\newtheorem{exmp}[thm]{Example}

\newtheorem*{clm*}{Claim}

\theoremstyle{remark}
\newtheorem{rem}[thm]{Remark}

\usepackage{todonotes}

\def\N{\mathbb{N}}

\def\E{\mathbb{E}}
\def\P{\mathbb{P}}
\def\R{\mathbb{R}}

\DeclareMathOperator{\Var}{Var}

\title{Quenched law of large numbers and quenched central limit theorem for multi-player leagues with ergodic strengths}

\date{  } %for removing the date use the empty command \date{ }

\author[1]{Jacopo Borga\thanks{\href{mailto:jacopo.borga@math.uzh.ch}{jacopo.borga@math.uzh.ch}}}
\author[1]{Benedetta Cavalli\thanks{\href{mailto:benedetta.cavalli@math.uzh.ch}{benedetta.cavalli@math.uzh.ch}}}

\affil[1]{Institut für Mathematik, Universität Zürich}

\makeatletter
\newcommand{\subjclass}[2][1991]{%
	\let\@oldtitle\@title%
	\gdef\@title{\@oldtitle\footnotetext{#1 \emph{Mathematics subject classification.} #2.}}%
}
\newcommand{\keywords}[1]{%
	\let\@@oldtitle\@title%
	\gdef\@title{\@@oldtitle\footnotetext{\emph{Key words and phrases.} #1.}}%
}
\makeatother

\keywords{Interacting random processes, invariance principles, sums of dependent and non-stationary random variables, multi-player competitions, modelization}

\subjclass[2010]{60F05, 60G50, 60K35}
\begin{document}

\maketitle

\begin{abstract}
We propose and study a new model for competitions, specifically sports multi-player leagues where the initial strengths of the teams are independent i.i.d.\ random variables that evolve during different days of the league according to independent ergodic processes. The result of each match is random: the probability that a team wins against another team is determined by a function of the strengths of the two teams in the day the match is played. 

Our model generalizes some previous models studied in the physical and mathematical  literature and is defined in terms of different parameters that can be statistically calibrated. We prove a quenched -- conditioning on the initial strengths of the teams -- law of large numbers and a quenched central limit theorem for the number of victories of a team according to its initial strength. 

To obtain our results, we prove a theorem of independent interest. For a stationary process $\bm \xi=(\bm \xi_i)_{i\in \N}$ satisfying a mixing condition and an independent sequence of i.i.d.\ random variables $(\bm s_i)_{i\in \N}$, we prove a quenched -- conditioning on $(\bm s_i)_{i\in\N}$ -- central limit theorem for sums of the form $\sum_{i=1}^{n}g\left(\bm \xi_i,\bm s_i\right)$, where $g$ is a bounded measurable function. We highlight that the random variables $g\left(\bm \xi_i,\bm s_i\right)$ are not stationary conditioning on $(\bm s_i)_{i\in\N}$.
\end{abstract}

\section{Introduction}

Multi-player competitions are a recurrent theme in physics, biology, sociology, and economics since they model several phenomena. We refer to the introduction of \cite{ben2007efficiency} for various examples of models of multi-player competitions related to evolution, social stratification, business world and other related fields. 

Among all these fields, sports are indubitably one of the most famous examples where modelling competitions is important. This is the case for two reasons: first, there is accurate and easily accessible data; second, sports competitions are typically head-to-head, and so they can be viewed as a nice model of multi-player competitions with binary interactions.

Sports multi-player leagues where the outcome of each game is completely deterministic have been studied for instance by Ben-Naim, Kahng, and Kim~\cite{ben2006dynamics} (see also references therein and explanations on how this model is related to urn models). Later, Ben-Naim and Hengartner~\cite{ben2007efficiency} investigated how randomness affects the outcome of multi-player competitions. In their model, the considered randomness is quite simple: they start with $N$ teams ranked from best
to worst and in each game they assume that the weaker team
can defeat the stronger team with a fixed probability. They investigate the total number of games needed for the best team to win the championship with high probability.

A more sophisticated model was studied by Bena\"{\i}m, Benjamini, Chen, and Lima~\cite{MR3346459}: they consider a finite connected graph, and place a team at each vertex of the graph. Two
teams are called a \emph{pair} if they share an edge. At discrete times, a match is
played between each pair of teams. In each match, one of the teams defeats the other (and gets a point) with
probability proportional to its current number of points raised to some fixed
power $\alpha>0$. They characterize the limiting behavior of the proportion of points
of the teams.

In this paper we propose a more realistic model for sport multi-player leagues that can be briefly described as follows (for a precise and formal description see \cref{sect:the_model}): we start with $2n$ teams having i.i.d.\ initial strengths. Then we consider a calendar of the league composed by $2n-1$ different days, and we assume that on each day each team plays exactly one match against another team in such a way that at the end of the calendar every team has played exactly one match against every other team (this coincides with the standard way how calendars are built in real football-leagues for the first half of the season). Moreover, we assume that on each day of the league, the initial strengths of the teams are modified by independent ergodic processes. Finally, we assume that a team wins against another team with probability given by a chosen function of the strengths of the two teams in the day the match is played.

We prove (see \cref{sect:main_results} for precise statements) a quenched law of large numbers and a quenched central limit theorem for the number of victories of a team according to its initial strength. Here quenched means that the results hold a.s.\ conditioning on the initial strengths of the teams in the league.

\subsection{The model}\label{sect:the_model}

We start by fixing some simple and standard notation.

\bigskip

\textbf{Notation.} 
Given $n\in\N=\{1,2,3,\dots\}$, we set $[n]=\{1,2,\dots,n\}$ and $[n]_0=\{0,1,2,\dots,n\}$. We also set $\R_+=\R\cap[0,\infty)$. 

We refer to random quantities using \textbf{bold} characters. Convergence in distribution is denoted by $\xrightarrow{d}$, almost sure convergence is denote by $\xrightarrow{a.s.}$, and convergence in probability by $\xrightarrow{P}$. 

Given a collection of sets $\mathcal A$, we denote by $\sigma\left(\mathcal A\right)$ and $\lambda\left(\mathcal A\right)$ the $\sigma$-algebra and the monotone class generated by $\mathcal A$, respectively. Given a random variable $\bm X$, we denote by $\sigma\left(\bm X\right)$  the $\sigma$-algebra generated by $\bm X$.

For any real-valued function $f$ we denote with $f^2$ the function such that $f^2(\cdot)=(f(\cdot))^2$.

\bigskip

\noindent We can now proceed with the precise description of our model for multi-player leagues. 

\bigskip

\noindent\textbf{The model.} We consider a league of $2n\in2\N$ teams denoted by $\{T_{i}\}_{i\in [2n-1]_0}$ whose \emph{initial random strengths} are denoted by $\{\bm s_{i}\}_{i\in [2n-1]_0}\in \R_+^{2n}$.

In the league every team $T_i$ plays $2n-1$ matches, one against each of the remaining teams $\{T_{j}\}_{j\in [2n-1]_0\setminus\{i\}}$. Note that there are in total ${2n}\choose{2}$ matches in the league. These matches are played in $2n-1$ different days in such a way that each team plays exactly one match every day.

For all $i\in[2n-1]_0$, the initial strength $\bm s_i$ of the team $T_i$ is modified every day according to a discrete time $\R_+$-valued stochastic process $\bm \xi^i=(\bm \xi^i_j)_{j\in \N}$. More precisely, the strength of the team $T_i$ on the $p$-th day is equal to $\bm s_i\cdot \bm \xi^i_p\in \R_+$.

We now describe the \emph{rules} for determining the winner of a match in the league. We fix a function $f:\mathbb R_+^2\to[0,1]$ that controls the winning probability of a match between two teams given their strengths. When a team with strength $x$ plays a match against another team with strength $y$, its probability of winning the match is equal to $f(x,y)$ and its probability of loosing is equal to $1-f(x,y)$ (we are excluding the possibility of having a draw). Therefore, if the match between the teams $T_i$ and $T_j$ is played the $p$-th day, then, conditionally on the random variables $\bm s_i, \bm s_j,\bm \xi^i_p,\bm \xi^j_p$, the probability that $T_i$ wins is $f(\bm s_i\cdot \bm \xi^i_p,\bm s_j\cdot \bm \xi^j_p)$.
Moreover, conditionally on the strengths of the teams, the results of different matches are independent.

\subsection{Goal of the paper}

The goal of this work is to study the model defined above when the number $2n$ of teams in the league is large. We want to look at the limiting behavior of the number of wins of a team with initial strength $s\in \R_+$ at the end of the league. More precisely, given $s\in \R_+$, we assume w.l.o.g.\ that the team $T_{0}$ has deterministic initial strength $s$, i.e.\ $\bm s_{0}=s$ a.s., and we set
\begin{equation}
 \bm W_n(s)\coloneqq\text{Number of wins of the team }T_{0}\text{ at the end of a league with $2n$ players}.
\end{equation}
We investigate a quenched law of large numbers and a quenched central limit theorem for $\bm W_n(s)$.
\subsection{Our assumptions}

In the following two subsections we state some assumptions on the model.

\subsubsection{Assumptions for the law of large numbers}\label{ass:LLN}

We make the following natural assumptions\footnote{The second hypothesis is not needed to prove our results but it is very natural for the model.} on the function $f:\mathbb R_+^2\to[0,1]\:$:
\begin{itemize}
	\item $f(x,y)$ is measurable;
	\item $f(x,y)$ is weakly-increasing in the variable $x$ and weakly-decreasing in the variable $y$.	
\end{itemize}
Recall also that it is not possible to have a draw, i.e.\ $f(x,y)+f(y,x)=1$, for all $x,y\in \R_+$.

Before describing our additional assumptions on the model, we introduce some further quantities.
Fix a Borel probability measure $\nu$ on $\R_+$; let $\bm \xi=(\bm \xi_\ell)_{\ell\in \N}$ be a discrete time $\R_+$-valued stochastic process such that 
\begin{equation}\label{eq:stationarity0}
	\bm \xi_\ell\stackrel{d}{=}\nu,\quad\text{for all}\quad \ell\in \N,
\end{equation}  
and\footnote{This is a weak-form of the \emph{stationarity property} for stochastic processes.}
\begin{equation}\label{eq:stationarity}
\left(\bm \xi_\ell,\bm \xi_k\right)\stackrel{d}{=} \left(\bm \xi_{\ell+\tau},\bm \xi_{k+\tau}\right), \quad\text{for all}\quad \ell,k,\tau\in\N.
\end{equation}

We further assume that the process $\bm \xi$ is \emph{weakly-mixing}, that is, for every $A \in \sigma(\bm \xi_1)$ and every collection of sets $B_\ell \in \sigma(\bm \xi_\ell)$, it holds that
\begin{equation}\label{eq:unif_weak_mix1}
\frac{1}{n} \sum_{\ell=1}^n \left|\P(A \cap B_\ell)-\P(A)\P(B_\ell)\right|\xrightarrow{n\to\infty} 0.
\end{equation}

The additional assumptions on our model are the following:

\begin{itemize}
\item For all $i\in[2n-1]_0$, the stochastic processes $\bm \xi^i$  are independent copies of $\bm \xi$.
\item The initial random strengths $\{\bm s_i\}_{i\in[2n-1]}$ of the teams different than $T_0$ are i.i.d.\ random variables on $\R_+$ with distribution $\mu$, for some Borel probability measure $\mu$ on $\R_+$.
\item The initial random strengths $\{\bm s_i\}_{i\in[2n-1]}$ are independent of the processes $\{\bm \xi^i\}_{i\in[2n-1]_0}$ and of the process  $\bm \xi$.
\end{itemize}

\subsubsection{Further assumptions for the central limit theorem}\label{ass:CLT}

In order to prove a central limit theorem, we need to make some stronger assumptions. The first assumption concerns the mixing properties of the process $\bm \xi$. For $k \in \mathbb{N}$, we introduce the two $\sigma$-algebras $\mathcal{A}_1^{k} = \sigma \left(\bm \xi_1, \dots, \bm \xi_k \right)$ and $\mathcal{A}_k^{\infty} = \sigma \left(\bm \xi_k, \dots \right)$ and we define for all $n\in\N$,
\begin{equation}\label{eq:def_alpha_n}
\alpha_n = \sup_{\substack{k \in \mathbb{N}\\ A \in \mathcal{A}_1^{k},B \in \mathcal{A}_{k+n}^{\infty}}} 
\left| \P \left( A \cap B \right) - \P(A)\P(B) \right|.
\end{equation}
We assume that 
\begin{equation}\label{eq:strongly_mix_plus}
\sum_{n=1}^{\infty} \alpha_n < \infty.
\end{equation}
Note that this condition, in particular, implies that the process $\bm \xi$ is \emph{strongly mixing}, that is, $\alpha_n \to 0$ as $n \to \infty$. 

Finally, we assume that there exist two sequences $p=p(n)$ and $q=q(n)$ such that: 
\begin{itemize}
\item  	$p\xrightarrow{n\to\infty} +\infty$ and $q\xrightarrow{n\to\infty} +\infty$,
\item 	$q=o(p)$ and $p=o(n)$ as $n \to \infty$,
\item   $ n p^{-1 } \alpha_q =o(1)$,
\item   $  \frac{p}{n}  \cdot \sum_{j=1}^p j \alpha_j = o(1)$.
\end{itemize}

\begin{rem}
The latter assumption concerning the existence of the sequences $p$ and $q$ is not very restrictive. For example, simply assuming that $\alpha_n = O(\frac{1}{n \log(n)})$ ensures that the four conditions are satisfied for $p=\frac{\sqrt n}{\log \log n}$ and $q=\frac{\sqrt n}{(\log \log n)^2}$. Indeed, in this case, the first two conditions are trivial, the fourth one follows by noting that $\sum_{j=1}^p j \alpha_j=O(p)$ thanks to the assumption in \cref{eq:strongly_mix_plus}, and finally the third condition follows by standard computations.
\end{rem}

\begin{rem}\label{rem:fbkwufobw}
	Note also that as soon as $p= O(\sqrt n)$ then the fourth condition is immediately verified.
	Indeed, $\sum_{j=1}^p j \alpha_j\leq \sqrt p \sum_{j=1}^{\sqrt p}  \alpha_j+p \sum_{j=\sqrt p}^{ p}  \alpha_j =o(p)$.
\end{rem}

\subsection{Main results}\label{sect:main_results}

\subsubsection{Results for our model of multi-player leagues}

Let $\bm{V},\bm{V'},\bm U, \bm U'$ be four independent random variables such that $\bm{V}\stackrel{d}{=}\bm{V'}\stackrel{d}{=}\nu$ and $\bm U\stackrel{d}{=}\bm U'\stackrel{d}{=}\mu$. Given a deterministic sequence $\vec{s}=(s_i)_{{i\in\N}}\in\R^{\N}_{+}$,  denote by $\P_{\vec{s}}$ the law of the random variable $\frac{\bm W_n(s)}{2n}$ when the initial strengths of the teams $(T_i)_{i\in[2n-1]}$ are equal to $\vec{s}=(s_i)_{{i\in[2n-1]}}$, i.e.\ we study $\frac{\bm W_n(s)}{2n}$ on the event
 $$\bm s_0=s\quad \text{ and }\quad(\bm s_i)_{{i\in[2n-1]}}=(s_i)_{{i\in[2n-1]}}.$$

\begin{thm}(Quenched law of large numbers)\label{thm:LLN}
	 Suppose that the assumptions in \cref{ass:LLN} hold. Fix any $s\in\R_+$. For $\mu^{\N}$-almost every sequence $\vec{s}=(s_i)_{{i\in\N}}\in\R^{\N}_{+}$, under $\P_{\vec{s}}$ the following convergence holds
	\begin{equation}
	\frac{\bm W_n(s)}{2n}\xrightarrow[n\to\infty]{P}\ell(s), 
	\end{equation}
	where 
	\begin{equation}
	\ell(s)=\E\left[f\left(s\cdot\bm{V},\bm U\cdot\bm{V'}\right)\right]=\int_{\R^3_+} f\left(s\cdot v,u\cdot v'\right)d\nu(v)d\nu(v')d\mu(u).
	\end{equation}
\end{thm}

We now state our second result.

\begin{thm}(Quenched central limit theorem)\label{thm:CLT}
	Suppose that the assumptions in \cref{ass:LLN} and \cref{ass:CLT} hold. Fix any $s\in\R_+$. For $\mu^{\N}$-almost every sequence $\vec{s}=(s_i)_{i\in\N}\in\R^{\N}_{+}$, under $\P_{\vec{s}}$ the following convergence holds
	\begin{equation}\label{eq:CLT}
	\frac{\bm W_n(s)- \E_{\vec{s}}[\bm W_n(s)] }{\sqrt{2n}}\xrightarrow{d} \bm{\mathcal{N}}\left(0,\sigma(s)^2 + \rho(s)^2\right), 
	\end{equation}
where, for $F_s(x,y)\coloneqq\E\left[f\left(s\cdot x,y\cdot \bm V'\right)\right]$ and $\tilde F_s \left(x,y\right) \coloneqq F_s\left(x,y \right) - \E[F_s\left(\bm V, y\right)]$,
\begin{equation}\label{eq:variance1_CLT}
	\sigma(s)^2=\E\left[F_s(\bm V,\bm U)-\left(F_s(\bm V,\bm U)\right)^2\right]=\ell(s)-\E\left[\left(F_s(\bm V,\bm U)\right)^2\right]
\end{equation}
and
\begin{equation}\label{eq:def_rho_s}
	\rho(s)^2=  \E \left[ \tilde F_s (\bm V, \bm U)^2 \right] + 2 \cdot \sum_{k=1}^{\infty} \E\left[ \tilde F_s(\bm{\xi}_1, \bm U) \tilde F_s (\bm{\xi}_{1+k}, \bm U') \right],
\end{equation}
the last series being convergent.
\end{thm}

\begin{rem}
	The assumption that the initial strengths of the teams $(\bm s_i)_{i\in\N}$ are independent could be relaxed from a theoretical point of view, but we believe that this is a very natural assumption for our model. 
\end{rem}

\subsubsection{Originality of our results and analysis of the limiting variance: a quenched CLT for functionals of ergodic processes and i.i.d.\ random variables}\label{sect:litterature_sum_var}
We now comment on the originality of our results and contextualize them in relation to the established  literature and prior studies on sums of dependent and not equi-distributed random variables. Additionally, we give some informal explanations  on the two components $\sigma(s)^2$ and $\rho(s)^2$ of the variance of the limiting Gaussian random variable in \cref{eq:CLT}. 

We start by noticing that, without loss of generality, we can assume that for every $j\in[2n-1]$, the team $T_{0}$ plays against the team $T_j$ the $j$-th day of the league. Denoting by $W_{j}=W_j(s)$ the event that the team $T_{0}$ wins against the team $T_j$, then $\bm W_n(s)$ rewrites as 
\begin{equation}
	\bm W_n(s)=\sum_{j=1}^{2n-1}\mathds{1}_{W_j}. 
\end{equation}
Note that the Bernoulli random variables $(\mathds{1}_{W_j})_{j\in[2n-1]}$ are only independent conditionally on the process $(\bm \xi^{0}_j)_{j\in[2n-1]}$.
In addition, under our assumptions, we have that the conditional parameters of the Bernoulli random variables are given by $\P_{\vec{s}} \left(W_j\middle | \bm \xi^{0}_j\right)= F_s (\bm \xi^{0}_j, s_j )$.
Therefore, under $\P_{\vec{s}}$, the random variable $\bm W_n(s)$ is a sum of Bernoulli  random variables that are \emph{neither independent nor identically distributed}. As a consequence, the proofs of the quenched law of large numbers (see \cref{sect:LLN}) and of the quenched central limit theorem (see \cref{sect:CLT}) do not follow from a simple application of classical results in the literature. 

We recall that it is quite simple to relax the identically distributed assumption in order to obtain a central limit theorem: the Lindeberg criterion, see for instance \cite[Theorem 27.2]{MR1324786}, gives a sufficient (and almost necessary) criterion for a sum of independent random
variables to converge towards a Gaussian random variable (after suitable renormalization). Relaxing independence is more delicate and there is no universal theory to do it. In the present paper, we combine two well-known theories to obtain our results: the theory for stationary ergodic processes (see for instance \cite{MR74711, MR0148125, MR0322926, MR1176496, MR2325294}) and the theory for $m$-dependent random variables (see for instance \cite{MR26771,MR350815,MR1747098}) and dependency graphs (see for instance \cite{MR681466, MR920273, MR1048950, MR2068873, MR4105789}).

To explain the presence of the two terms in the variance, we start by noticing that using the law of total conditional variance, the conditional variance of $\mathds{1}_{W_j}$ is given by $\Var_{\vec{s}} \left(\mathds{1}_{W_j}\middle | \bm \xi^{0}_j\right)= F_s (\bm \xi^{0}_j, s_j ) - (F_s (\bm \xi^{0}_j, s_j ))^2$. The term $\sigma(s)^2$ arises as the limit of 
$$ \frac{1}{2n} \sum_{j=1}^{2n-1} \Var_{\vec{s}} \left(\mathds{1}_{W_j}\middle | \bm \xi^{0}_j\right),$$
and this is in analogy with the case of sums of independent but not identically distributed Bernoulli random variables. The additional term $\rho(s)^2$, on the contrary, arises from the fluctuations of the conditional parameters of the Bernoulli variables, i.e.\  $\rho(s)^2$ comes from the limit of 
\begin{equation}\label{eq:fbewbfw}
	\frac{1}{2n} \sum_{j=1}^{2n-1} \Var \left(\P_{\vec{s}} \left(W_j\middle | \bm \xi^{0}_j\right) \right)=\frac{1}{2n} \sum_{j=1}^{2n-1} \Var \left( F_s (\bm \xi^{0}_j, s_j ) \right).
\end{equation}
Note that an additional difficulty arises from the fact that the sums in the last two equations are not independent (but we will show that they are asymptotically independent).

To study the limit in \cref{eq:fbewbfw} we prove in \cref{sect:CLT} the following general result that we believe to be of independent interest.

\begin{thm}\label{prop:clt_sum_of_the_g}
	Suppose that the assumptions in \cref{eq:stationarity0,eq:stationarity}, and in \cref{ass:CLT} hold. Let $g:\R^2_+\to\R_+$ be a bounded, measurable function, and define $\tilde g :\R^2_+\to\R$ by $\tilde g (x,y) \coloneqq g(x,y) - \E \left[ g\left( \bm V, y\right) \right]$. Then, the quantity
	\begin{equation}\label{eq:def_of_rho}
		\rho_g^2 \coloneqq  \E \left[ \tilde g (\bm V, \bm U)^2 \right] + 2 \cdot \sum_{k=1}^{\infty} \E\left[ \tilde g(\bm{\xi}_1, \bm U) \tilde g (\bm{\xi}_{1+k}, \bm U') \right]
	\end{equation}
	is finite. Moreover, for $\mu^{\N}$-almost every sequence $(s_i)_{{i\in\N}}\in\R^{\N}_{+}$, the following convergence holds
	\begin{equation}\label{eq:CLT2}
		\frac{\sum_{j=1}^{2n-1}\tilde g\left(\bm \xi_j,s_j\right)}{\sqrt{2n}}\xrightarrow{d} \bm{\mathcal{N}}(0,\rho_g^2).
	\end{equation}
\end{thm}

The main difficulty in studying the sum $\sum_{j=1}^{2n-1}g\left(\bm \xi_j,s_j\right)$ is that fixing a deterministic realization $(s_i)_{{i\in\N}}\in\R^{\N}_{+}$ imposes that the random variables $g\left(\bm \xi_j,s_j\right)$ are not stationary anymore.

In many of the classical results available in the literature (see references given above) -- in addition to moment and mixing conditions -- stationarity is assumed, and so we cannot directly apply these results. An exception are \cite{MR1492353,MR3257385}, where stationarity is not assumed but it is assumed a stronger version\footnote{This is assumption (B3) in \cite{MR3257385} that is also assumed in \cite{MR1492353}. Since in our case all moments of $g\left(\bm \xi_j,s_j\right)$ are finite (because $g$ is a bounded function), we can take the parameter $\delta$ in \cite{MR3257385} equal to $+\infty$ and thus condition (B3) in \cite{MR3257385} requires that $\sum_{n=1}^{\infty} n^2 \alpha_n <+\infty$.} of our condition in \cref{eq:strongly_mix_plus}. Therefore, to the best of our knowledge, \cref{prop:clt_sum_of_the_g} does not follow from known results in the literature.

\subsection{Calibration of the parameters of the model and some examples}\label{sect:param_and_examples}

An interesting feature of our model consists in the fact that the main parameters that characterize the evolution of the league can be statistically calibrated in order for the model to describe real-life tournaments. Such parameters are:

\begin{itemize}
\item The function $f:\mathbb R_+^2\to[0,1]$ that controls the winning probability of the matches.
\item The distribution $\mu$ of the initial strengths $(\bm s_i)_i$ of the teams.
\item The marginal distribution $\nu$ of the tilting process $\bm \xi$.
\end{itemize}

We end this section with two examples. The first one is more theoretical and the second one more related to the statistical calibration.

\begin{exmp}\label{exmp:league}
	We assume that:
	\begin{itemize}
		\item $f(x,y)=\frac{x}{x+y}$;
		\item for all $i\in\N$, the initial strengths $\bm s_i$ are uniformly distributed in $[0,1]$, i.e.\ $\mu$ is the Lebesgue measure on $[0,1]$;
		\item the tilting process $\bm \xi$ is a Markov chain with state space $\{a,b\}$ for some $a\in(0,1),b\in(1,\infty)$, with transition matrix $\begin{pmatrix}
		p_a & 1-p_a \\
		1-p_b & p_b 
		\end{pmatrix}$ for some $p_a,p_b\in(0,1)$. Note that the invariant measure $\nu=(\nu_a,\nu_b)$ is equal to $	\left(\frac{1-p_b}{2-p_a-p_b},\frac{1-p_a}{2-p_a-p_b}\right)$.
	\end{itemize} 
	Under this assumptions, \cref{thm:LLN} guarantees that for any fixed $s\in\R_+$ and for almost every realization $\vec{s}=( s_i)_{{i\in\N}}$ of the random sequence $(\bm s_i)_{{i\in\N}}\in[0,1]^{\N}$, the following convergence holds under $\P_{\vec{s}}$:
	\begin{equation}
	\frac{\bm W_n(s)}{2n}\xrightarrow[n\to\infty]{P}\ell(s), 
	\end{equation}
	where 
	\begin{equation}\label{eq:exemp_l_S}
	\ell(s)=\int_{\R^3_+} \frac{s v}{s v+u v'}d\nu(v)d\nu(v')d\mu(u)=\sum_{i,j\in\{a,b\}} \frac{s\cdot i}{j}\log\left(1+\frac{j}{s\cdot i}\right)\nu_i\cdot \nu_j.
	\end{equation}
	In particular if $a=\frac{1}{2}, b=2, p_a=\frac 1 2, p_b=\frac 1 2$ then
	\begin{equation}\label{eq:expression_exemp_ls}
	\ell(s)=\frac{s}{16}\log\left(\frac{(1+s)^8(4+s)(1+4s)^{16}}{2^{32}\cdot s^{25}}\right),
	\end{equation}
	whose graph is plotted in \cref{fig:simple_exemp}.
	\begin{figure}[htbp]
		\centering
		\includegraphics[scale=.5]{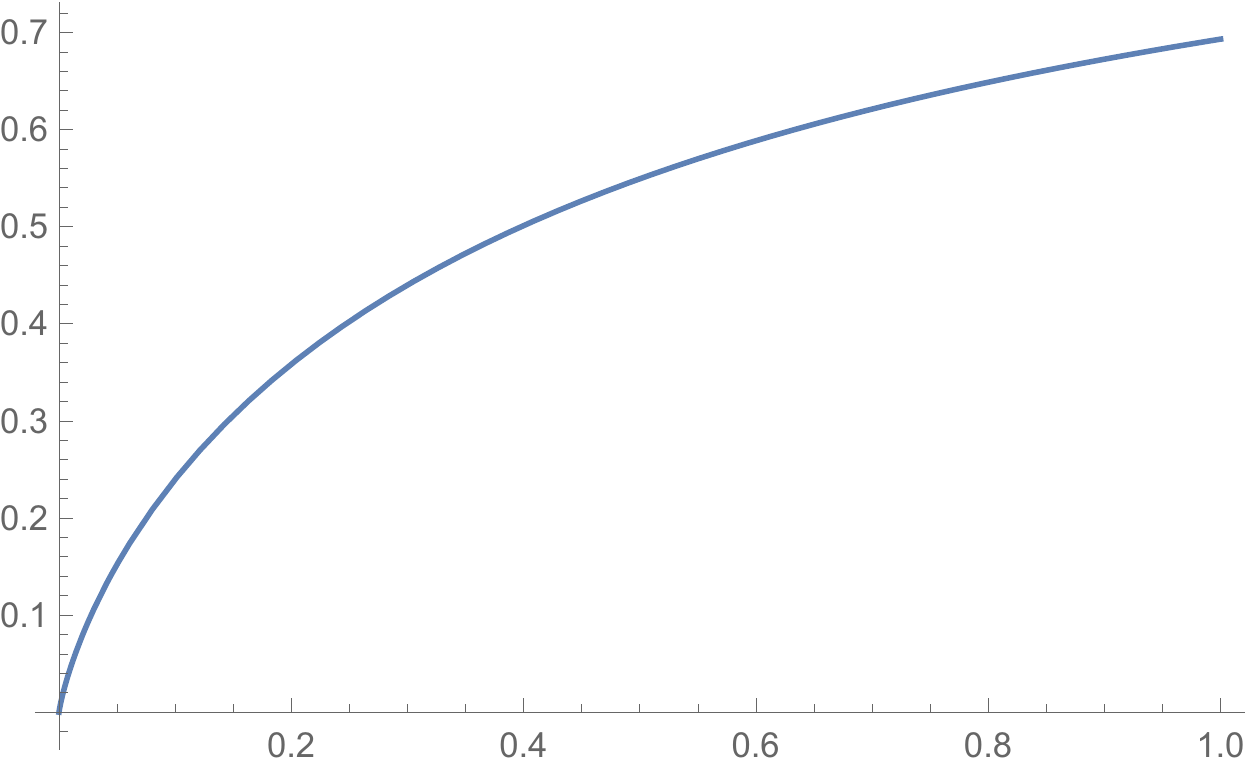}
		\caption{The graph of the function $\ell(s)$ in \cref{eq:expression_exemp_ls} for $s\in[0,1]$. \label{fig:simple_exemp}}
	\end{figure}

	In addition, \cref{thm:CLT} implies that for any $s\in\R_+$ and for almost every realization $\vec{s}=( s_i)_{{i\in\N}}$ of the random sequence $(\bm s_i)_{{i\in\N}}\in[0,1]^{\N}$, the following convergence also holds under $\P_{\vec{s}}$: 
	\begin{equation}
	\frac{\bm W_n(s)- \E_{\vec{s}}[\bm W_n(s)] }{\sqrt{2n}}\xrightarrow{d} \bm{\mathcal{N}}(0,\sigma(s)^2 + \rho(s)^2),
	\end{equation}
	where $\sigma(s)^2$ and $\rho(s)^2$ can be computed as follows. From \cref{eq:variance1_CLT,eq:def_rho_s} we know that
	\begin{equation}
		\sigma(s)^2=\ell(s)-\E\left[\left(F_s(\bm V,\bm U)\right)^2\right]
	\end{equation}
	and that
	\begin{equation}
		\rho(s)^2=  \E \left[ \tilde F_s (\bm V, \bm U)^2 \right] + 2 \cdot \sum_{k=1}^{\infty} \E\left[ \tilde F_s(\bm{\xi}_1, \bm U) \tilde F_s (\bm{\xi}_{1+k}, \bm U') \right].
	\end{equation}
	Recall that $\ell(s)$ was computed in \cref{eq:exemp_l_S}. Note that $$F_s(x,y)=\E\left[f\left(s\cdot x,y\cdot \bm V'\right)\right]=\sum_{v'\in \{a,b\}}f\left(s\cdot x,y\cdot v'\right)\nu_{v'}$$
	and that
	$$\tilde F_s(x,y)=F_s(x,y)-\E\left[F_s(\bm V,y)\right]=\sum_{v'\in\{a,b\}}f\left(s\cdot x,y\cdot v'\right)\nu_{v'}-\sum_{(v,v')\in \{a,b\}^2}f\left(s\cdot v,y\cdot v'\right)\nu_{v'}\nu_{v}.$$
	Therefore
	\begin{equation}
		\E\left[\left(F_s(\bm V,\bm U)\right)^2\right]=\sum_{v\in \{a,b\}}\left(\int_0^1\left(\sum_{v'\in\{a,b\}}f\left(s\cdot v,u\cdot v'\right)\nu_{v'}\right)^2du\right)\nu_{v}
	\end{equation}
	and 
	\begin{equation}
	\E\left[\left(\tilde F_s(\bm V,\bm U)\right)^2\right]=\sum_{w\in \{a,b\}}\left(\int_0^1\left(\sum_{v'\in\{a,b\}}f\left(s\cdot w,u\cdot v'\right)\nu_{v'}-\sum_{(v,v')\in \{a,b\}^2}f\left(s\cdot v,u\cdot v'\right)\nu_{v'}\nu_{v}\right)^2du\right)\nu_{v}.
	\end{equation}
	Moreover,
	\begin{equation}
	\sum_{k=1}^{\infty} \E\left[ \tilde F_s(\bm{\xi}_1, \bm U) \tilde F_s (\bm{\xi}_{1+k}, \bm U') \right]=	\sum_{k=1}^{\infty}  \sum_{(i,j)\in \{a,b\}^2}  \P(\bm{\xi}_1=i,\bm{\xi}_{1+k}=j) \int_0^1 \tilde F_s(i, u)\; du \int_0^1 \tilde F_s (j, u') \; du'.
	\end{equation}
	It remains to compute $\P(\bm{\xi}_1=i,\bm{\xi}_{1+k}=j)$ for $i,j\in \{a,b\}$. Note that
	\begin{equation}
	\begin{pmatrix}
	p_a & 1-p_a \\
	1-p_b & p_b 
	\end{pmatrix}=
	\begin{pmatrix}
	1 & \frac{1-p_a}{p_b-1} \\
	1 & 1
	\end{pmatrix}
	\begin{pmatrix}
	1 & 0 \\
	0 & (p_a+p_b-1)
	\end{pmatrix}
	\begin{pmatrix}
	\frac{p_b-1}{p_a+p_b-2} & \frac{p_a-1}{p_a+p_b-2} \\
	\frac{1-p_b}{p_a+p_b-2} & \frac{p_b-1}{p_a+p_b-2} 
	\end{pmatrix}
	\eqqcolon SJS^{-1}.
	\end{equation}
	Therefore, 
	\begin{align}
	&\P(\bm{\xi}_1=a,\bm{\xi}_{1+k}=a)=\frac{p_b-1 + (p_a-1) (p_a + p_b-1)^k}{p_a + p_b-2}\cdot \nu_a,\\
	&\P(\bm{\xi}_1=a,\bm{\xi}_{1+k}=b)=\frac{(p_a-1) ((p_a + p_b-1)^k-1)}{p_a + p_b-2}\cdot \nu_a,\\
	&\P(\bm{\xi}_1=b,\bm{\xi}_{1+k}=a)=\frac{(p_b-1)((p_a + p_b-1)^k-1)}{p_a + p_b-2}\cdot \nu_b,\\
	&\P(\bm{\xi}_1=b,\bm{\xi}_{1+k}=b)=\frac{p_a-1 + (p_b-1) (p_a + p_b-1)^k}{p_a + p_b-2}\cdot \nu_b.
	\end{align}
	With some tedious but straightforward computations\footnote{We developed a \emph{Mathematica} software to quickly make such computations for various choices of the function $f(x,y)$, and of the parameters $a,b,p_a,p_b$. The software is available at the following \href{https://drive.google.com/drive/folders/1CXZVpe-HJvtJNNGlThu2J-PO9OHme0KP?usp=sharing}{link}.}, we can explicitly compute $\sigma(s)^2$ and $\rho(s)^2$.
	The graphs of the two functions $\sigma(s)^2$ and $\rho(s)^2$ for $s\in[0,1]$ are plotted in \cref{fig:diagram_variance} for three different choices of the parameters $a,b,p_a,p_b$. It is interesting to note that $\sigma(s)^2$ is much larger than $\rho(s)^2$ when $p_a$ and $p_b$ are small, $\sigma(s)^2$ is comparable with $\rho(s)^2$ when $p_a$ and $p_b$ are around 0.9, and $\rho(s)^2$ is much larger than $\sigma(s)^2$ when $p_a$ and $p_b$ are very  close to 1.
	
	\begin{figure}[htbp]
		\centering
		\includegraphics[scale=.4]{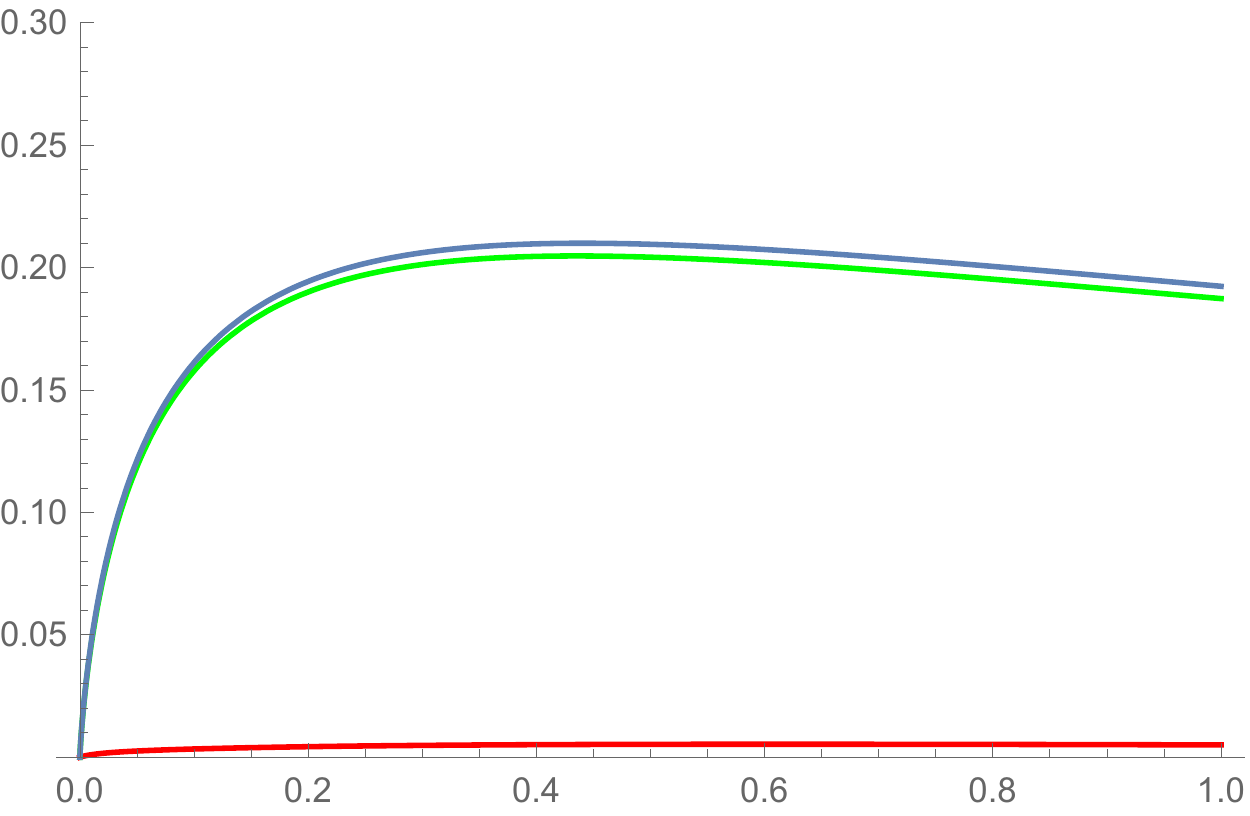}
		\hspace{0.1cm}
		\includegraphics[scale=.4]{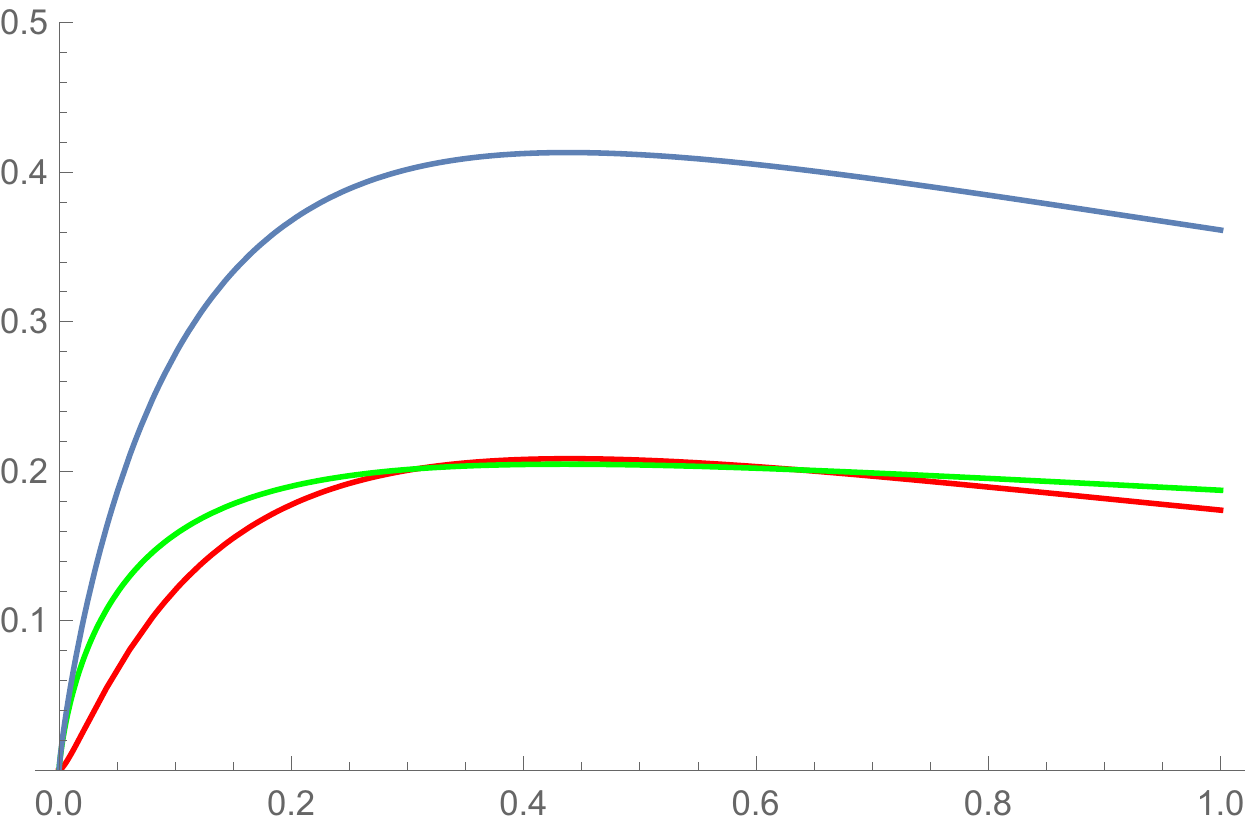}
		\hspace{0.1cm}
		\includegraphics[scale=.4]{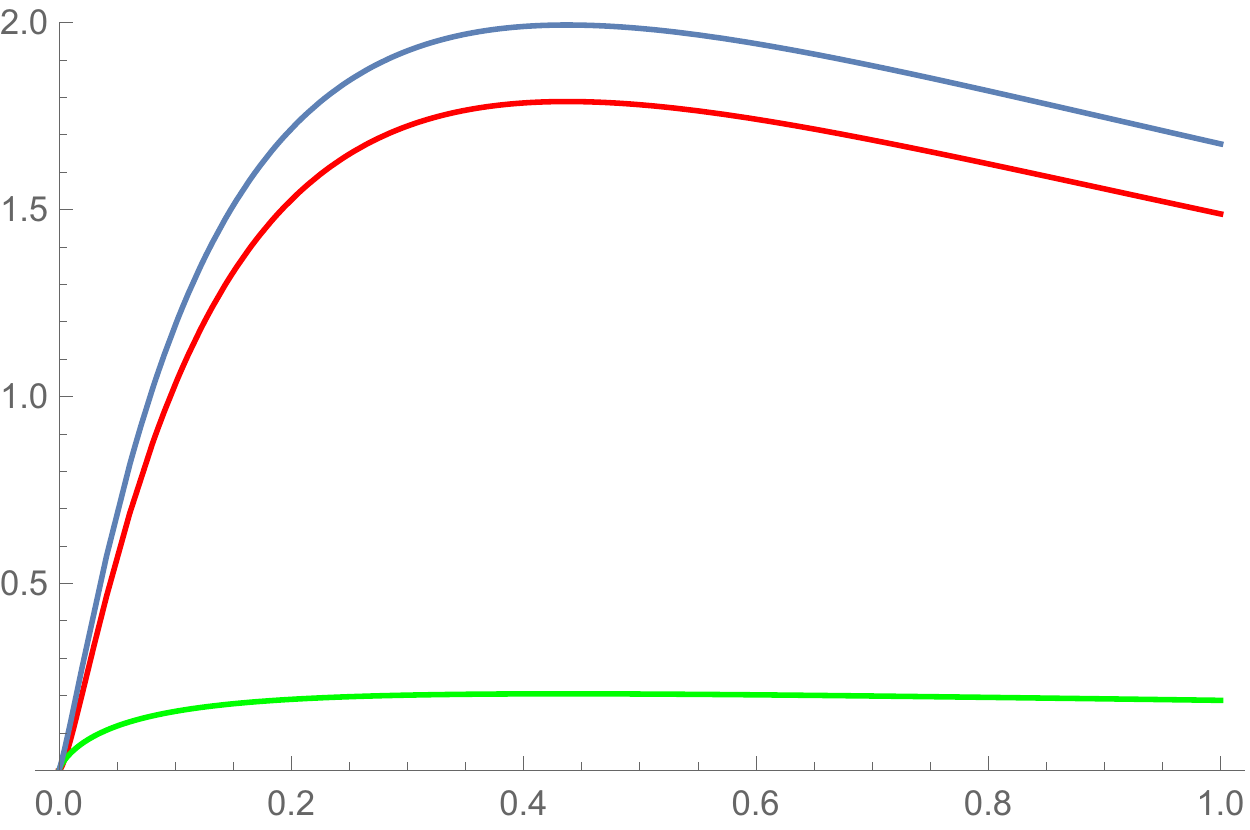}
		\caption{In green the graph of $\sigma(s)^2$. In red the graph of $\rho(s)^2$. In blue the graph of $\sigma(s)^2+\rho(s)^2$.
			\textbf{Left:} The parameters of the model are $a=1/2,b=2,p_a=2/5,p_b=2/5$.
			\textbf{Middle:} The parameters of the model are $a=1/2,b=2,p_a=92/100,p_b=92/100$.
			\textbf{Right:} The parameters of the model are $a=1/2,b=2,p_a=99/100,p_b=99/100$.  \label{fig:diagram_variance}}
	\end{figure}	
\end{exmp}

\begin{exmp}	
	We collect here some data related to the Italian national basketball league in order to compare some real data with our theoretical results. We believe that it would be interesting to develop a more accurate and precise analysis of real data in some future projects.
	
	In \cref{fig:table_basket}, the rankings of the last 22 national leagues played among exactly 16 teams in the league are shown (some leagues, like the 2011-12 league, are not tracked since in those years the league was not formed by 16 teams). 
	\begin{figure}[htbp]
		\centering
		\includegraphics[scale=0.95]{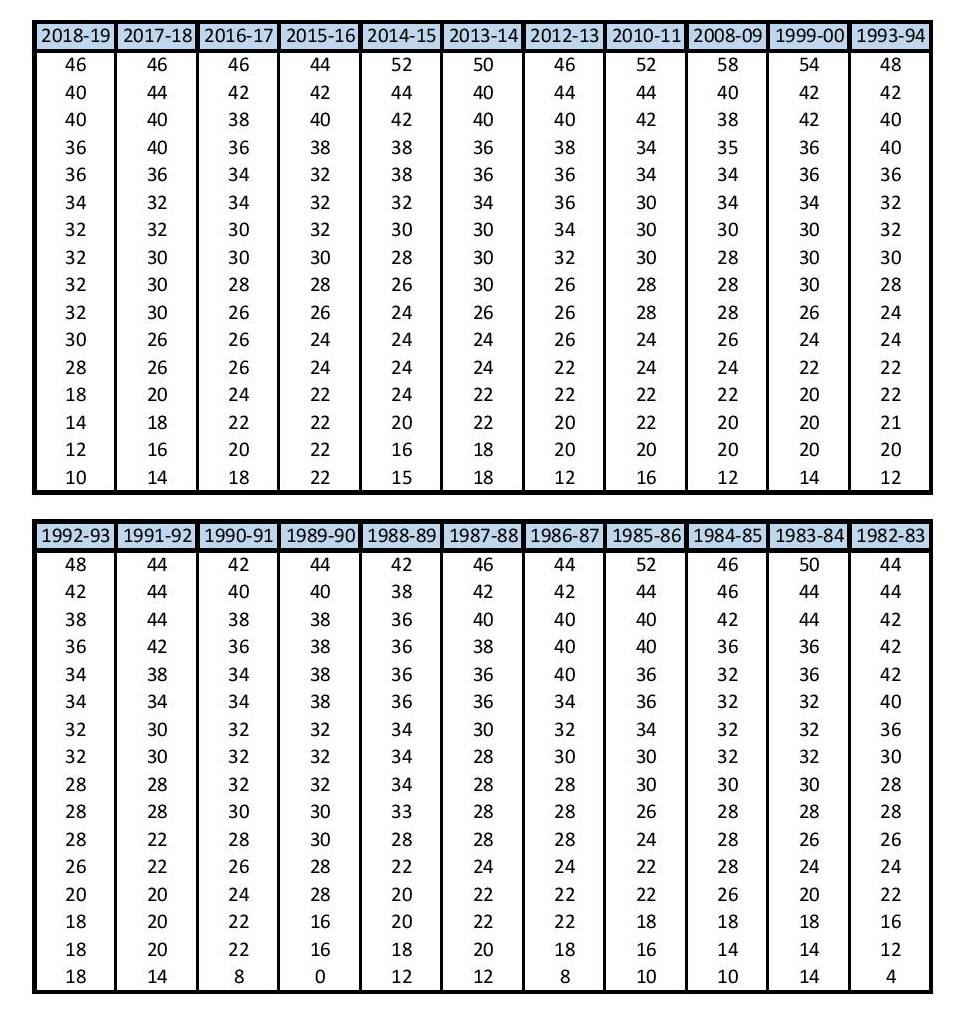}
		\caption{The rankings of the last 22 national italian basket leagues played with exactly 16 teams in the league. In the Italian basket league every team plays two matches against every other team and every victory gives two points.\label{fig:table_basket}}
	\end{figure}
	The mean number of points of the 22 collected leagues are (from the team ranked first to the team ranked 16-th):
	\begin{multicols}{4}
		\begin{enumerate}
			\item 47,45
			\item 42,27
			\item 40,18
			\item 37,59
			\item 35,91
			\item 34,09
			\item 31,73
			\item 30,55
			\item 29,18
			\item 27,77
			\item 26,09
			\item 24,36
			\item 22,00
			\item 19,59
			\item 17,82
			\item 12,41
		\end{enumerate}
	\end{multicols}

	The diagram of this ranking is given in the left-hand side of \cref{fig:diagram_basket}. 
	
	\medskip
	
	As mentioned above, an interesting question consists in assessing whether it is possible to describe the behaviour of these leagues by using our model. More precisely, we looked for a function $f(x,y)$ and two distributions $\mu$ and $\nu$ such that the graph of $\ell(s)$ for $s\in[0,1]$ can well approximate the graph in the left-hand side of \cref{fig:diagram_basket}. We find out that choosing $\mu$ to be the uniform measure on the interval $[0.1,0.999]$, $\nu=0.6\cdot\delta_{0.25}+0.9\cdot\delta_{1.3},$ and
	\begin{equation}\label{eq:guess_f}
	f(x,y)\coloneqq\frac{g(x)}{g(x)+g(y)},\quad\text{with}\quad g(x)\coloneqq\log \left(1-\min\left\{\frac{x}{1.3},0.999\right\}\right),
	\end{equation}
	then the graph of $\ell(s)$ for $s\in[0.1,0.999]$, is the one given in the right-hand side of \cref{fig:diagram_basket}. Note that the two graphs have a similar convexity.
	\begin{figure}[htbp]
		\centering
		\includegraphics[scale=.55]{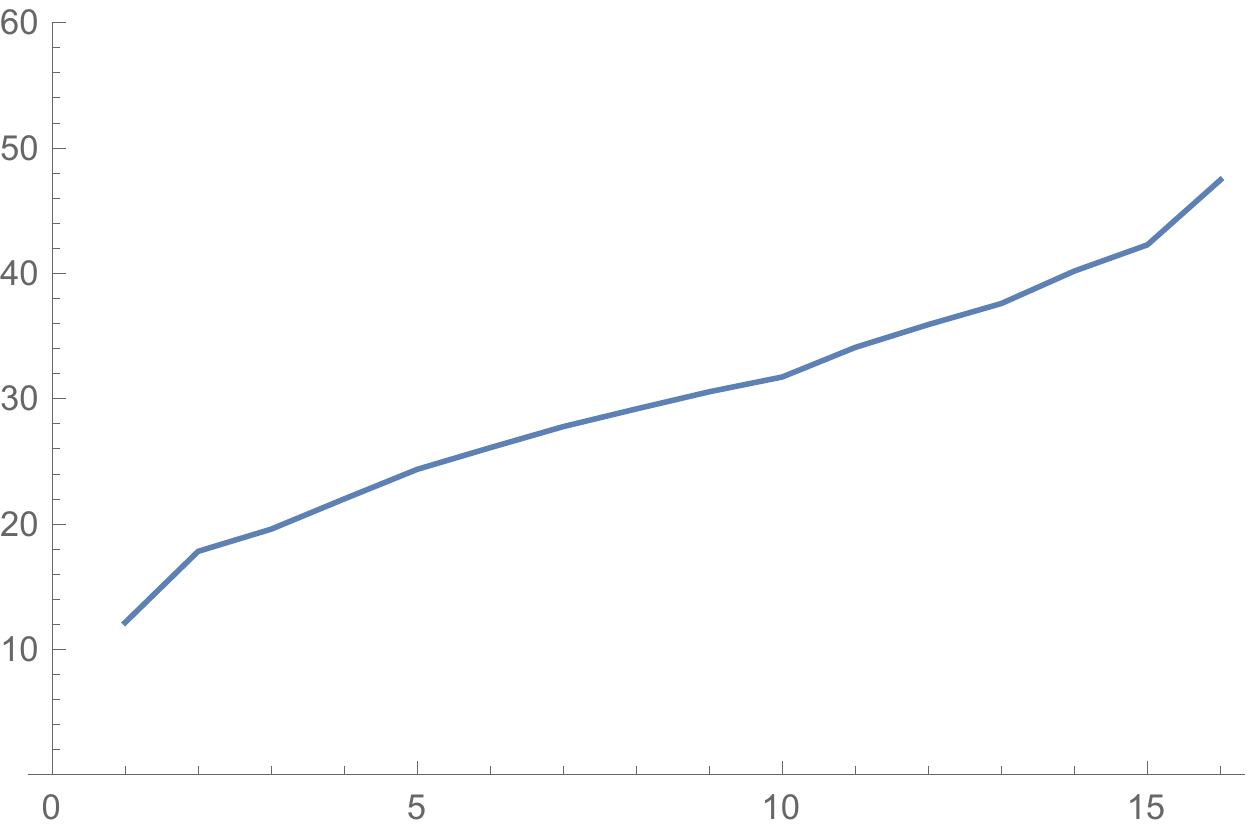}
		\hspace{1cm}
		\includegraphics[scale=.55]{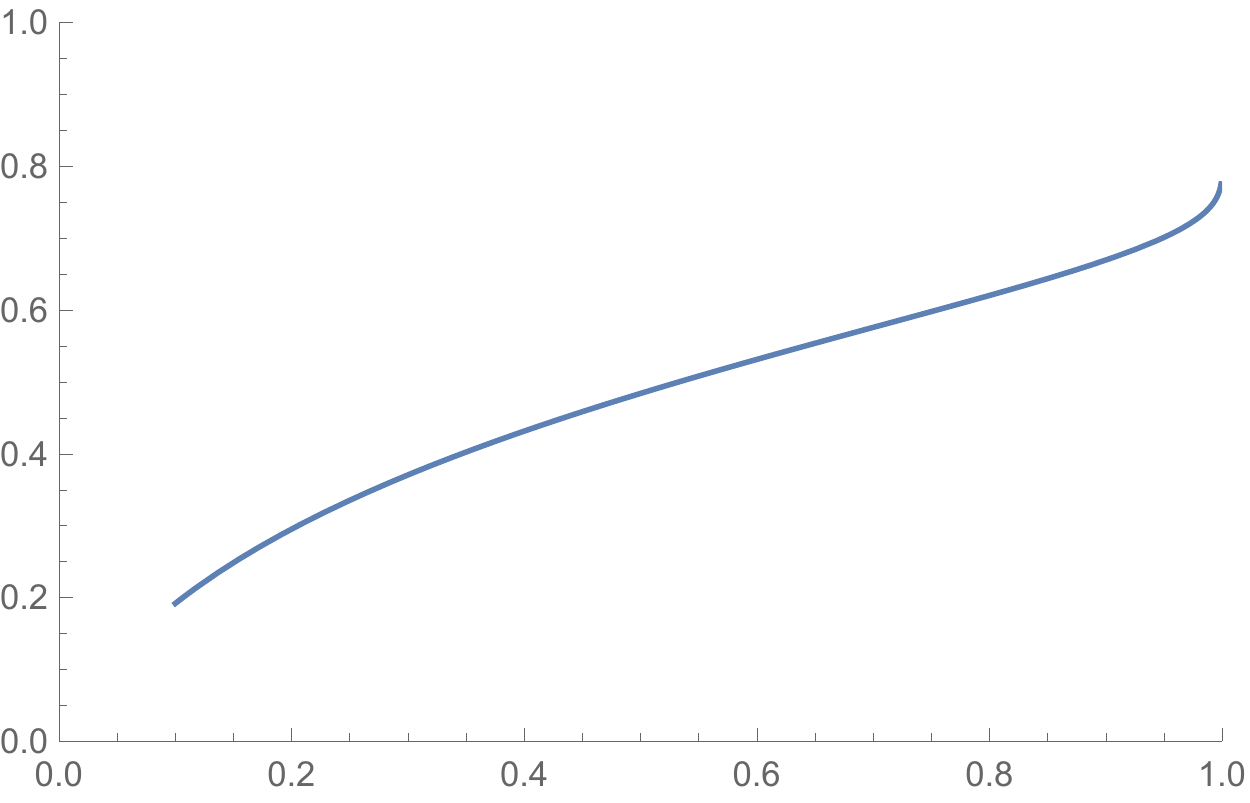}
		\caption{\textbf{Left:} The diagram of the mean number of points of the 22 leagues collected in \cref{fig:table_basket}. On the $x$-axis the teams are ranked from the weakest to the strongest; on the $y$-axis the mean number of points are plotted.
			\textbf{Right:} The graph of $\ell(s)$ for $s\in[0.1,0.999]$ for the specific $\mu,\nu$ and $f(x,y)$ given in (and before) \cref{eq:guess_f}. \label{fig:diagram_basket}}
	\end{figure}
\end{exmp}

\subsection{Open problems}

We collect some open questions and conjectures that we believe might be interesting to investigate in future projects:

\begin{itemize}
	\item Conditioning on the initial strengths of the teams, how many times do we need to run the league in order to guarantee that the strongest team a.s.\ wins the league? We point out that a similar question was investigated in \cite{ben2007efficiency} for the model considered by the authors.

	\item In the spirit of large deviations results, we believe it would be interesting to study the probability that the weakest team wins the league, again conditioning on the initial strengths of the teams.

	\item Another natural question is to investigate the whole final ranking of the league. We conjecture the following.
	
	For a sequence of initial strengths $(s_i)_{i\in [2n-1]_0}$, we denote by  $(\tilde T_i)_{i\in [2n-1]_0}$ the sequence of teams reordered according to their initial strengths $(s_i)_{i\in [2n-1]_0}$ (from the weakest to the strongest). Set
	\begin{equation}
	\tilde{\bm W}_n(i)\coloneqq\text{Number of wins of the team }\tilde T_{i}\text{ at the end of a league with $2n$ players}.
	\end{equation}
	Let  $\tilde{\bm {\mathcal W}}_n(x):[0,1]\to \R$ denote the piece-wise linear continuous function obtained by interpolating the values $\tilde{\bm W}_n(i/n)$ for all $i\in [2n-1]_0$.

	%In order to simplify the following statement, we also assume that the measure $\mu$ satisfies $\mu=h(x)dx$ for some integrable function $h:\R_+\to\R$. 
	Denote by $H_{\mu}(y)$ the cumulative distribution function of $\mu$ and by $H_{\mu}^{-1}(x)$ the generalized inverse distribution function, i.e.\ $H_{\mu}^{-1}(x)=\inf\{y\in \R_{+}: H_{\mu}(y)\geq x \}$.
	\begin{conj}
		Suppose that the assumptions in \cref{ass:LLN} hold with the additional requirement that the function $f$ is continuous\footnote{The assumption that $f$ is continuous might be relaxed.}. For $\mu^{\N\cup\{0\}}$-almost every sequence $\vec{s}=(s_i)_{{i\in\N\cup\{0\}}}\in\R^{\N\cup\{0\}}_{+}$ and for every choice of the calendar of the league, under $\P_{\vec{s}}$ the following convergence of càdlàg processes holds
		\begin{equation}
		\frac{\tilde{\bm {\mathcal W}}_n(x)}{2n}\xrightarrow[n\to\infty]{P}\ell(H_{\mu}^{-1}(x)), 
		\end{equation}
		where $\ell(s)$ is defined as in \cref{thm:LLN} by
		\begin{equation}
		\ell(s)=\E\left[f\left(s\cdot\bm{V},\bm U\cdot\bm{V'}\right)\right]=\int_{\R^3_+} f\left(s\cdot v,u\cdot v'\right)d\nu(v)d\nu(v')d\mu(u).
		\end{equation}
	\end{conj}
	
	Note that the correlations between various teams in the league strongly depend on the choice of the calendar but we believe that this choice does not affect the limiting result in the conjecture above. We refer the reader to \cref{fig:sim_whole_league} for some simulations that support our conjecture.
	
	We also believe that the analysis of the local limit (as defined in \cite{MR4055194}) of the whole ranking should be an interesting but challenging question (and in this case we believe that the choice of the calendar will affect the limiting object). Here, with local limit we mean the limit of the ranking induced by the $k$ teams -- for every fixed $k\in\N$ -- in the neighborhood (w.r.t.\ the initial strengths) of a distinguished team (that can be selected uniformly at random or in a deterministic way, say for instance the strongest team).
	
	\item As mentioned above, we believe that it would be interesting to develop a more accurate and precise analysis of real data in order to correctly calibrate the parameters of our model collected at the beginning of \cref{sect:param_and_examples}.

\end{itemize}

\begin{figure}[htbp]
	\centering
	\includegraphics[scale=.4]{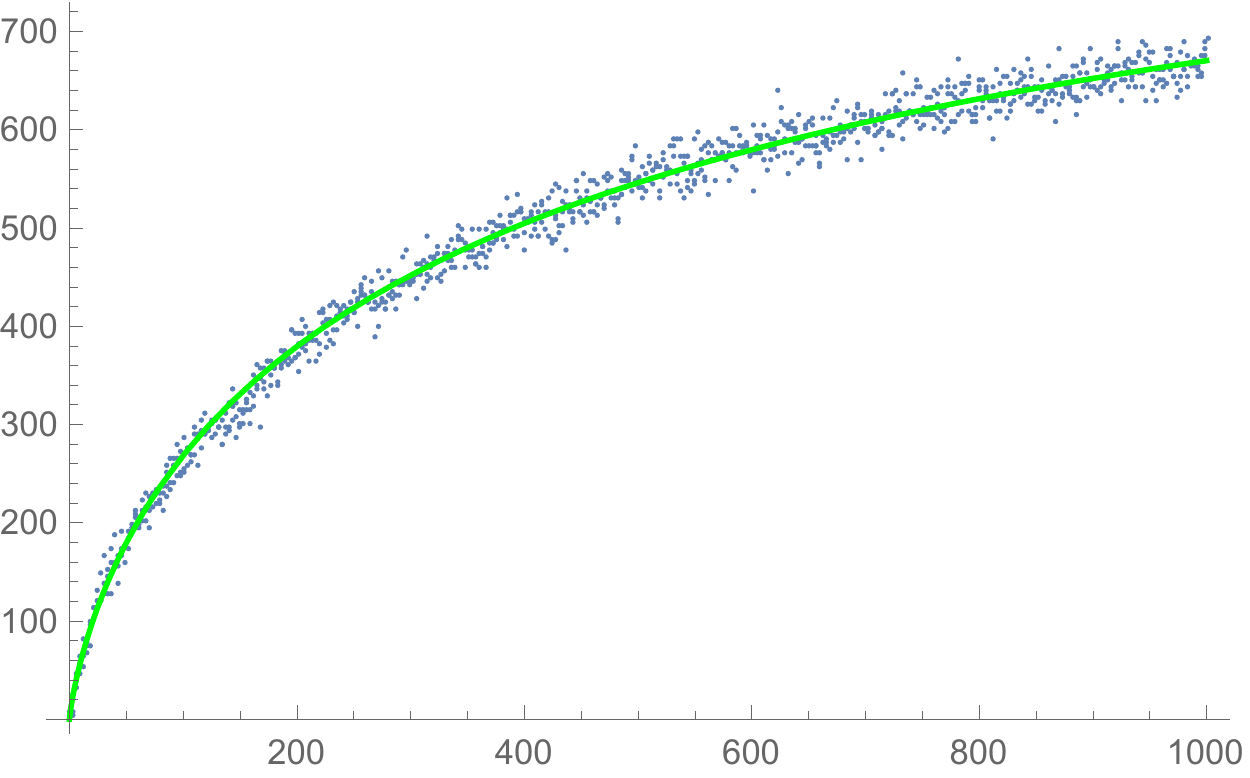}
	\hspace{0.1cm}
	\includegraphics[scale=.4]{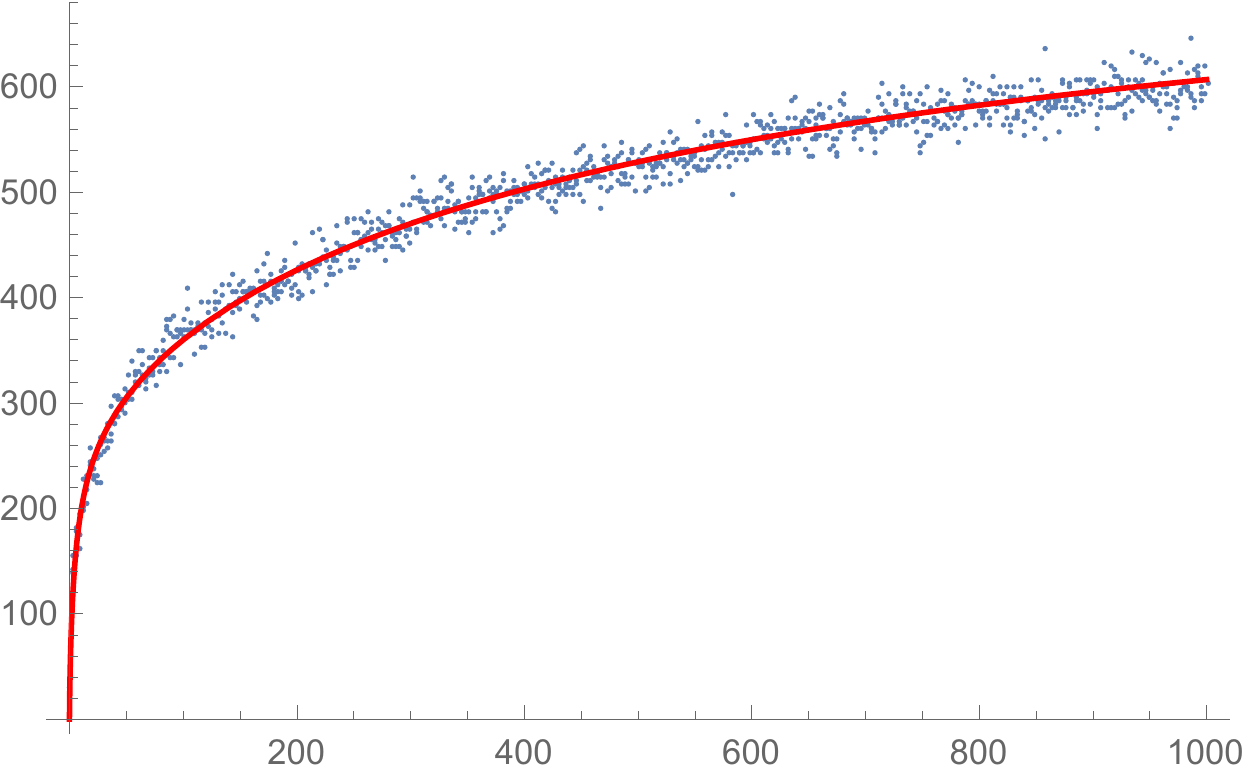}
	\hspace{0.1cm}
	\includegraphics[scale=.4]{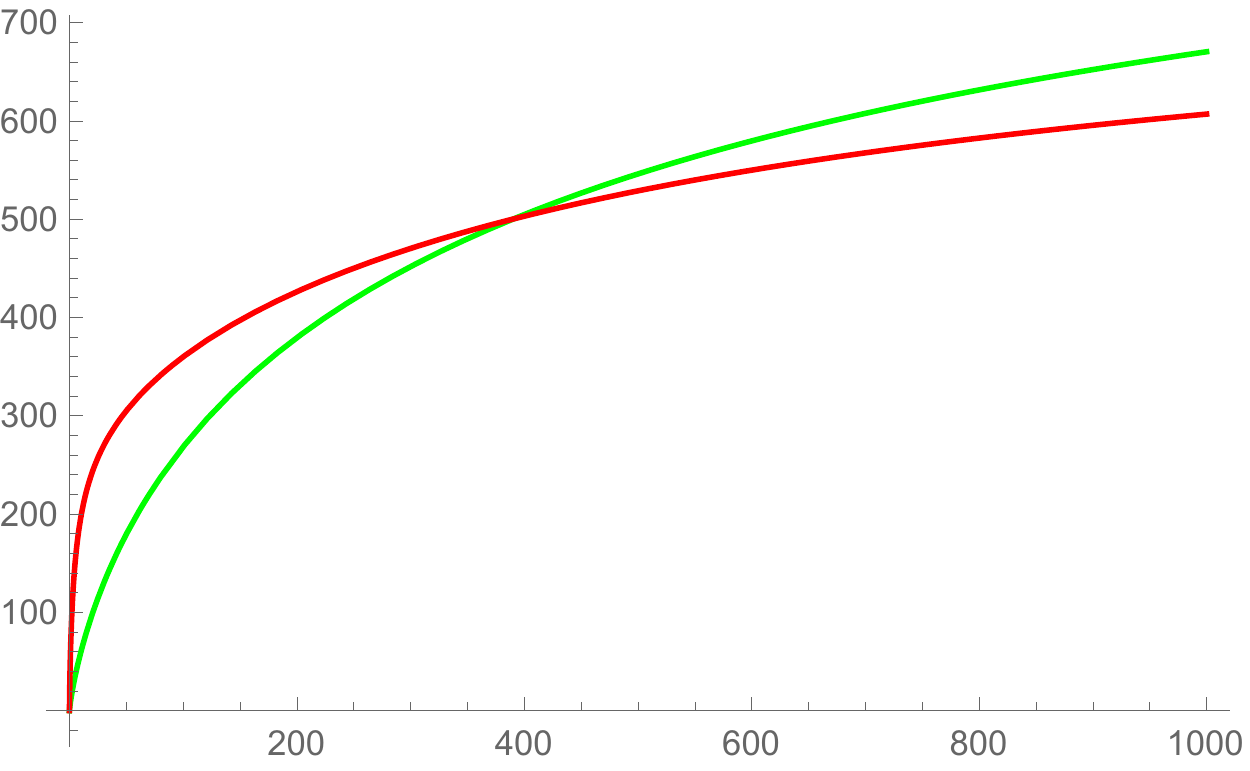}
	\caption{Two simulations of the diagrams of $(\tilde{\bm W}_n(i))_{i\in[999]_0}$ in the setting of \cref{exmp:league} with different choices of $a, b, p_a$ and $p_b$. The values $(\tilde{\bm W}_n(i))_{i\in[999]_0}$ are plotted in blue. The limiting functions $1000\cdot\ell(H_{\mu}^{-1}(x/1000))$ for $x\in[0,1000]$ are plotted in green and red respectively.  \textbf{Left:} In this simulation the parameters are $a=\frac{1}{2}, b=2, p_a=\frac 1 2, p_b=\frac 1 2$.
	\textbf{Middle:} In this simulation the parameters are $a=\frac{1}{10}, b=10, p_a=\frac{1}{10}, p_b=\frac{1}{10}$.
	\textbf{Right:} The two limiting functions are overlapped in the same diagram to highlight the different slopes. 
	\label{fig:sim_whole_league}}
\end{figure}

\section{Proof of the law of large numbers}\label{sect:LLN}

The proof of \cref{thm:LLN} follows from the following result using standard second moment arguments.
\begin{prop}\label{prop:first_mom}
We have that
\begin{equation}
	\E\left[\frac{\bm W_n(s)}{2n}\middle| (\bm s_i)_i\right]\xrightarrow[n\to\infty]{a.s.} \ell(s) ,
\qquad\text{and}\qquad
	\E\left[\left(\frac{\bm W_n(s)}{2n}\right)^2\middle| (\bm s_i)_i\right]\xrightarrow[n\to\infty]{a.s.} \ell(s)^2. 
\end{equation}
\end{prop}
The rest of this section is devoted to the proof of \cref{prop:first_mom}.

\medskip

We recall (see \cref{sect:litterature_sum_var}) that we assumed that, for every $j\in[2n-1]$, the team $T_{0}$ plays against the team $T_j$ the $j$-th day and we denoted by $W_{j}=W_j(s)$ the event that the team $T_{0}$ wins the match. In particular, $\bm W_n(s)=\sum_{j=1}^{2n-1}\mathds{1}_{W_j}$ and
\begin{equation}\label{eq:prob_win_match}
	\P\left(W_j\middle | \bm s_j, \bm \xi^{0}_j,\bm \xi^{j}_j\right)=f(s\cdot \bm \xi^{0}_j,\bm s_j \cdot \bm \xi^{j}_j).
\end{equation}

\begin{proof}[Proof of \cref{prop:first_mom}]
	We start with the computations for the first moment. From \cref{eq:prob_win_match} we have that
	\begin{equation}\label{eq:evfwryibfweonfpiwe}
		\E\left[\bm W_n(s)\middle| (\bm s_i)_i\right]=\sum_{j=1}^{2n-1}\P(W_j| (\bm s_i)_i)=\sum_{j=1}^{2n-1}\E\left[f\left(s\cdot \bm \xi^{0}_j,\bm s_j\cdot \bm \xi^j_j\right)\middle| \bm s_j\right].
	\end{equation}
	Since for all $j\in [2n-1]$, $\bm \xi^j_j$, $\bm \xi^{0}_j$ and $\bm s_j$ are independent, $\bm \xi^{0}_j\stackrel{d}{=}\bm V$, and $\bm \xi^j_j\stackrel{d}{=}\bm V'$, we have
	\begin{equation}
		\E\left[f\left(s\cdot \bm \xi^{0}_j,\bm s_j\cdot \bm \xi^j_j\right)\middle| \bm s_j\right]=G_s\left(\bm s_j\right),
	\end{equation}
	where $G_s(x)=\E\left[f\left(s\cdot \bm V,x\cdot \bm V'\right)\right]$.
	By the Law of large numbers, we can conclude that
	\begin{equation}\label{eq:evfwuitgwrefbwofnwryibfweonfpiwe}
		\E\left[\frac{\bm W_n(s)}{2n}\middle| (\bm s_i)_i\right]=\frac{1}{2n}\sum_{j=1}^{2n-1}G_s\left(\bm s_j\right)
		\xrightarrow[n\to\infty]{a.s.} \E\left[G_s\left(\bm U\right)\right]=\ell(s).
	\end{equation}

We now turn to the second moment. 
Since for all $i,j\in [2n-1]$ with $i\neq j$, conditioning on $(\bm s_i,\bm s_j, \bm \xi^{0}_i,\bm \xi^{i}_i,\bm \xi^{0}_j,\bm \xi^{j}_j)$ the events $W_i$ and $W_j$ are independent, we have that, using \cref{eq:prob_win_match},
\begin{multline}
\E\left[\bm W_n(s)^2\middle| (\bm s_i)_i\right]=\E\left[\bm W_n(s)\middle| (\bm s_i)_i\right]+\sum_{\substack{\ell,j=1\\ \ell\neq j}}^{2n-1}\P(W_\ell\cap W_j | \bm s_\ell,\bm s_j)\\
=\E\left[\bm W_n(s)\middle| (\bm s_i)_i\right]+\sum_{\substack{\ell,j=1\\ \ell\neq j}}^{2n-1}\E\left[f\left(s\cdot \bm \xi^{0}_\ell,\bm s_\ell\cdot \bm \xi^\ell_\ell\right)f\left(s\cdot \bm \xi^{0}_j,\bm s_j\cdot \bm \xi^j_j\right)\middle| \bm s_\ell,\bm s_j\right].
\end{multline}
For all $\ell,j\in [2n-1]$ with $\ell\neq j$, $\bm s_\ell$, $\bm s_j$, $\bm \xi^\ell_\ell$ and $\bm \xi^j_j$ are mutually independent and independent of $(\bm \xi^{0}_\ell,\bm \xi^{0}_j)$. In addition, $(\bm \xi^{0}_\ell,\bm \xi^{0}_j)\stackrel{d}{=}(\bm \xi_\ell,\bm \xi_j)$ and $\bm \xi^\ell_\ell\stackrel{d}{=}\bm \xi^j_j\stackrel{d}{=}\bm V'$. Thus, we have that
\begin{equation}
\E\left[f\left(s\cdot \bm \xi^{0}_\ell,\bm s_\ell\cdot \bm \xi^\ell_\ell\right)f\left(s\cdot \bm \xi^{0}_j,\bm s_j\cdot \bm \xi^j_j\right)\middle| \bm s_\ell,\bm s_j\right]
=\E\left[F_s\left(\bm \xi_\ell,\bm s_\ell\right)F_s\left(\bm \xi_j,\bm s_j\right)\middle| \bm s_\ell,\bm s_j\right],
\end{equation}
where we recall that  $F_s(x,y)=\E\left[f\left(s\cdot x,y\cdot \bm V'\right)\right]$.
Simple consequences of the computations done for the first moment are that 
$$\frac{\E\left[\bm W_n(s)\middle| (\bm s_i)_i\right]}{n^2}\xrightarrow[n\to\infty]{a.s.} 0\quad\text{and}\quad\frac{\E\left[\sum_{j=1}^{2n-1} F_s^2 (\bm \xi_j, \bm s_j)\middle| (\bm s_i)_i\right]}{n^2}\xrightarrow[n\to\infty]{a.s.} 0.$$ Therefore, we can write
\begin{multline}\label{eq:fbnwiruefbeownfw}
\E\left[\left(\frac{\bm W_n(s)}{2n}\right)^2\middle| (\bm s_i)_i\right]=
\E\left[\frac{1}{4n^2}\sum_{\ell,j=1}^{2n-1}F_s\left(\bm \xi_\ell,\bm s_\ell\right) F_s\left(\bm \xi_j,\bm s_j\right)\middle| (\bm s_i)_i\right]+\bm o(1)\\
= 
\E\left[\left(\frac{1}{2n}\sum_{j=1}^{2n-1} F_s\left(\bm \xi_j,\bm s_j\right)\right)^2\middle| (\bm s_i)_i\right]+\bm o(1),
\end{multline}
where $\bm o(1)$ denotes a sequence of random variables that a.s.\ converges to zero.
We now need the following result, whose proof is postponed at the end of this section.
\begin{prop} \label{prop:conv_for_bnd_cont_funct}
	For all bounded, measurable functions $g:\R^2_+\to\R_+$, we have that
	\begin{equation}\label{eq:second_mom_funct}
		\E\left[\left(\frac{1}{2n}\sum_{j=1}^{2n-1}g\left(\bm \xi_j,\bm s_j\right)\right)^2\middle| (\bm s_i)_i\right]\xrightarrow[n\to\infty]{a.s.}\E\left[g\left(\bm V,\bm U\right)\right]^2.
	\end{equation}
\end{prop}

From \cref{eq:fbnwiruefbeownfw} and the proposition above, we conclude that
\begin{equation}
\E\left[\left(\frac{\bm W_n(s)}{2n}\right)^2\middle| (\bm s_i)_i\right]\xrightarrow[n\to\infty]{a.s.}\E\left[F_s\left(\bm V,\bm U\right)\right]^2=\ell(s)^2.
\end{equation}
This concludes the proof of \cref{prop:first_mom}.
\end{proof}

It remains to prove \cref{prop:conv_for_bnd_cont_funct}. We start with the following preliminary result.

\begin{lem} \label{lem:first_sec_mom_for_ret}
For every quadruple $(A,A',B,B')$ of Borel subsets of $\R_+$, we have that
\begin{equation}
\E\left[\frac{1}{4n^2}\sum_{\ell,j=1}^{2n-1} \mathds{1}_{A \times B}\left(\bm \xi_\ell,\bm s_\ell\right) \mathds{1}_{A' \times B'}\left(\bm \xi_j,\bm s_j\right)\middle| (\bm s_i)_i\right]\xrightarrow[n\to\infty]{a.s.} \E\left[\mathds{1}_{A \times B}(\bm V,\bm U)\right]\E\left[\mathds{1}_{A' \times B'}(\bm V,\bm U)\right]\label{eq:second_mom_ind_ret}.
\end{equation}
\end{lem}

\begin{proof}
Note that since the process $\bm \xi$ is independent of $(\bm s_i)_i$,
\begin{multline}\label{eq:rewriting_the_expression}
\E\left[\frac{1}{4n^2}\sum_{\ell,j=1}^{2n-1} \mathds{1}_{A \times B}\left(\bm \xi_\ell,\bm s_\ell\right) \mathds{1}_{A' \times B'}\left(\bm \xi_j,\bm s_j\right)\middle| (\bm s_i)_i\right]\\
=\frac{1}{2n}\sum_{\ell=1}^{2n-1}\mathds{1}_{B}(\bm s_\ell) \cdot \frac{1}{2n}\sum_{j=1}^{2n-1} \P\left(\bm \xi_\ell\in A,\bm \xi_j\in A'\right)\mathds{1}_{B'}(\bm s_j).
\end{multline} 
For all $\ell\in[2n-1]$, we can write
\begin{align}\label{eq:split_sum_prob}
\frac{1}{2n}\sum_{j=1}^{2n-1} \P\left(\bm \xi_\ell\in A,\bm \xi_j\in A'\right)&\mathds{1}_{B'}(\bm s_j)\\
=\P\left(\bm V \in A\right)&\P\left(\bm V \in A'\right)\frac{1}{2n}\sum_{j=1}^{2n-1} \mathds{1}_{B'}(\bm s_j)\\
&+\frac{1}{2n}\sum_{j=1}^{2n-1} \left(\P\left(\bm \xi_\ell\in A,\bm \xi_j\in A'\right)-\P\left(\bm V \in A\right)\P\left(\bm V \in A'\right)\right)\mathds{1}_{B'}(\bm s_j).
\end{align}
First, from the Law of large numbers we have that
\begin{equation}\label{eq:first_lim_real}
\frac{1}{2n}\sum_{j=1}^{2n-1}\mathds{1}_{B'}(\bm s_j)\xrightarrow[n\to\infty]{a.s.}\mu(B').
\end{equation}
Secondly, we show that the second sum in the right-hand side of \cref{eq:split_sum_prob} is negligible. We estimate
\begin{multline}
\left|\frac{1}{2n}\sum_{j=1}^{2n-1} \left(\P\left(\bm \xi_\ell\in A,\bm \xi_j\in A'\right)-\P\left(\bm V \in A\right)\P\left(\bm V \in A'\right)\right)\mathds{1}_{B'}(\bm s_j)\right|\\
\leq
\frac{1}{2n}\sum_{j=1}^{2n-1} \left|\P\left(\bm \xi_\ell\in A,\bm \xi_j\in A'\right)-\P\left(\bm V\in A\right)\P\left(\bm V \in A'\right)\right|\\
=
\frac{1}{2n}\sum_{j=1}^{\ell-1} \left|\P\left(\bm \xi_\ell\in A,\bm \xi_j\in A'\right)-\P\left(\bm V\in A\right)\P\left(\bm V \in A'\right)\right|\\
+\frac{1}{2n}\sum_{j=\ell}^{2n-1} \left|\P\left(\bm \xi_\ell\in A,\bm \xi_j\in A'\right)-\P\left(\bm V\in A\right)\P\left(\bm V \in A'\right)\right|
\end{multline}
Using the stationarity assumption in \cref{eq:stationarity}, the right-hand side of the equation above can be rewritten as follows
\begin{multline}
\frac{1}{2n}\sum_{j=2}^{\ell} \left|\P\left(\bm \xi_j\in A,\bm \xi_1\in A'\right)-\P\left(\bm V\in A\right)\P\left(\bm V \in A'\right)\right|\\
+\frac{1}{2n}\sum_{j=1}^{2n-\ell} \left|\P\left(\bm \xi_1\in A,\bm \xi_j\in A'\right)-\P\left(\bm V\in A\right)\P\left(\bm V \in A'\right)\right|.
\end{multline}
Therefore, we obtain that
\begin{multline}
\left|\frac{1}{2n}\sum_{j=1}^{2n-1} \left(\P\left(\bm \xi_\ell\in A,\bm \xi_j\in A'\right)-\P\left(\bm V \in A\right)\P\left(\bm V \in A'\right)\right)\mathds{1}_{B'}(\bm s_j)\right|\\
\leq
\frac{1}{2n}\sum_{j=1}^{2n} \left|\P\left(\bm \xi_1\in A',\bm \xi_j\in A\right)-\P\left(\bm V\in A\right)\P\left(\bm V \in A'\right)\right|\\
+\frac{1}{2n}\sum_{j=1}^{2n} \left|\P\left(\bm \xi_1\in A,\bm \xi_j\in A'\right)-\P\left(\bm V\in A\right)\P\left(\bm V \in A'\right)\right|.
\end{multline}
The upper bound above is independent of $\ell$ and tends to zero because the process $\bm \xi$ is weakly-mixing (see \cref{eq:unif_weak_mix1}); we can thus deduce from \cref{eq:split_sum_prob,eq:first_lim_real} that, uniformly for all $\ell\in[2n-1]$,
\begin{equation}\label{eq:ifbuewbfoewnfoiewnfew}
\frac{1}{2n}\sum_{j=1}^{2n-1} \P\left(\bm \xi_\ell\in A,\bm \xi_j\in A'\right)\mathds{1}_{B'}(\bm s_j)
\xrightarrow[n\to\infty]{a.s.}
\P\left(\bm V\in A\right)\P\left(\bm V \in A'\right)\mu(B').
\end{equation}
Hence, from \cref{eq:rewriting_the_expression,eq:ifbuewbfoewnfoiewnfew} and the Law of large numbers, we can conclude that
\begin{multline}
\E\left[\frac{1}{4n^2}\sum_{\ell,j=1}^{2n-1} \mathds{1}_{A \times B}\left(\bm \xi_\ell,\bm s_\ell\right) \mathds{1}_{A' \times B'}\left(\bm \xi_j,\bm s_j\right)\middle| (\bm s_i)_i\right]\\
\xrightarrow[n\to\infty]{a.s.} \P\left(\bm V\in A\right)\mu(B)\P\left(\bm V \in A'\right)\mu(B')
=\E\left[\mathds{1}_{A \times B}(\bm V,\bm U)\right]\E\left[\mathds{1}_{A' \times B'}(\bm V,\bm U)\right].\qedhere
\end{multline}
\end{proof}

We now show that we can extend the result of \cref{lem:first_sec_mom_for_ret} to all Borel subsets of $\R_+^2$, denoted by  $\mathcal{B}(\R_+^2)$. 
We also denote by $\mathcal{R}(\R_+^2)$ the sets of rectangles $A\times B$ of $\R_+^2$ with $A,B\in\mathcal{B}(\R_+)$.

We recall that we denote by $\sigma\left(\mathcal A\right)$ and $\lambda\left(\mathcal A\right)$ the sigma algebra and  the monotone class generated by a collection of sets $\mathcal A$, respectively.

\begin{lem}\label{lem:ind_fnct}
For all $C,C'\in\mathcal{B}(\R_+^2)$, we have that 
\begin{equation}
\E\left[\frac{1}{4n^2}\sum_{\ell,j=1}^{2n-1}\mathds{1}_{C}\left(\bm \xi_\ell,\bm s_\ell\right)\mathds{1}_{C'}\left(\bm \xi_j,\bm s_j\right)\middle| ( \bm s_i)_i\right]\xrightarrow[n\to\infty]{a.s.} \E\left[\mathds{1}_{C}(\bm V,\bm U)\right]\E\left[\mathds{1}_{C'}(\bm V,\bm U)\right]\label{eq:second_mom_ind_bor}.
\end{equation}
\end{lem}

\begin{proof}
We first fix a rectangle $A\times B\in\mathcal{R}(\R_+^2)$ and we consider the set
\begin{equation}
\mathcal{A}_{A\times B}\coloneqq \left\{C\in\mathcal{B}(\R_+^2)\middle |\text{Eq. }\eqref{eq:second_mom_ind_bor} \text{ holds with } C'=A\times B \right\}.
\end{equation}
By \cref{lem:first_sec_mom_for_ret}, we have $\mathcal{R}(\R_+^2)\subseteq\mathcal{A}_{A\times B}\subseteq \mathcal{B}(\R_+^2)$.  If we show that $\mathcal{A}_{A\times B}$ is a monotone class, then we can conclude that $\mathcal{A}_{A\times B}=\mathcal{B}(\R_+^2)$. Indeed, by the monotone class theorem (note that $\mathcal{R}(\R_+^2)$ is closed under finite intersections), we have that
\begin{equation}\label{eq:class_inclusions}
\mathcal{B}(\R_+^2)= \sigma\left(\mathcal{R}(\R_+^2)\right)=\lambda\left(\mathcal{R}(\R_+^2)\right) \subseteq\mathcal{A}_{A\times B}.
\end{equation}
The equality $\mathcal{A}_{A\times B}=\mathcal{B}(\R_+^2)$ implies that \cref{eq:second_mom_ind_bor} holds for every set in $\mathcal{B}(\R_+^2)\times \mathcal{R}(\R_+^2)$. Finally, if we also show that for any fixed Borel set $C^*\in\mathcal{B}(\R_+^2)$, the set
\begin{equation}
\mathcal{A}_{C^*}\coloneqq \left\{C'\in\mathcal{B}(\R_+^2)\middle |\text{Eq. }\eqref{eq:second_mom_ind_bor} \text{ holds with } C=C^* \right\}
\end{equation}
is a monotone class, then using again the same arguments that we used in \cref{eq:class_inclusions} (note that $\mathcal{R}(\R_+^2)\subseteq\mathcal{A}_{C^*}$ thanks to the previous step) we can conclude that \cref{eq:second_mom_ind_bor} holds for every pair of sets in $\mathcal{B}(\R_+^2)\times \mathcal{B}(\R_+^2)$, proving the lemma. Therefore, in order to conclude the proof, it is sufficient to show that $\mathcal{A}_{A \times B}$ and $\mathcal{A}_{C^*}$ are monotone classes.

\medskip

We start by proving that $\mathcal{A}_{A\times B}$ is a monotone class:

\begin{itemize}
\item Obviously $\R_+^2\in \mathcal{A}_{A\times B}$.
\item If $C,D\in\mathcal{A}_{A\times B}$ and $C\subseteq D$, then $D\setminus C\in \mathcal{A}_{A\times B}$ because $\mathds{1}_{D\setminus C}=\mathds{1}_{D}-\mathds{1}_{C}$. 
\item Let now $(C_m)_{m\in\N}$ be a sequence of sets in $\mathcal{A}_{A\times B}$ such that $C_m \subseteq C_{m+1}$ for all $m\in\N$. We want to show that $C\coloneqq\bigcup_m C_m\in \mathcal{A}_{A\times B}$, i.e.\ that 
\begin{equation}\label{eq:ifbueiwbfowebnfoewinf}
\E\left[\frac{1}{4n^2}\sum_{\ell,j=1}^{2n-1}\mathds{1}_{C}\left(\bm \xi_\ell,\bm s_\ell\right)\mathds{1}_{A\times B}\left(\bm \xi_j,\bm s_j\right)\middle| ( \bm s_i)_i\right]
\xrightarrow[n\to\infty]{a.s.} 
\E\left[\mathds{1}_{C}(\bm V,\bm U)\right]\E\left[\mathds{1}_{A \times B}(\bm V,\bm U)\right].
\end{equation}
Since $\mathds{1}_{C}=\lim_{m\to \infty}\mathds{1}_{C_m}$, then by monotone convergence we have for all $n\in \N$,
\begin{multline}\label{eq:lim_of_indicator}
\E\left[\frac{1}{4n^2}\sum_{\ell,j=1}^{2n-1}\mathds{1}_{C_m}\left(\bm \xi_\ell,\bm s_\ell\right)\mathds{1}_{A\times B}\left(\bm \xi_j,\bm s_j\right)\middle| ( \bm s_i)_i\right]
\xrightarrow[m\to\infty]{a.s.}\\
\frac{1}{4n^2}\sum_{\ell,j=1}^{2n-1}\E\left[\mathds{1}_{C}\left(\bm \xi_\ell,\bm s_\ell\right)\mathds{1}_{A\times B}\left(\bm \xi_j,\bm s_j\right)\middle|  \bm s_\ell,\bm s_j\right].
\end{multline}
We also claim that:
\begin{itemize}
	\item[(a)] For all $m\in\N$, 
	\begin{equation}
	\frac{1}{4n^2}\sum_{\ell,j=1}^{2n-1}\E\left[\mathds{1}_{C_m}\left(\bm \xi_\ell,\bm s_\ell\right)\mathds{1}_{A\times B}\left(\bm \xi_j,\bm s_j\right)\middle|  \bm s_\ell,\bm s_j\right]\xrightarrow[n\to\infty]{a.s.} \E\left[\mathds{1}_{C_m}(\bm V,\bm U)\right]\E\left[\mathds{1}_{A \times B}(\bm V,\bm U)\right].
	\end{equation}
	\item[(b)] The convergence in \cref{eq:lim_of_indicator} holds uniformly for all $n\in\N$.
\end{itemize}
Item (a) holds since $C_m \in \mathcal{A}_{A\times B}$.
Item (b) will be proved at the end. Items (a) and (b) allow us to exchange the following a.s.-limits as follows 
\begin{align}
\lim_{n\to \infty}\E&\left[\frac{1}{4n^2}\sum_{\ell,j=1}^{2n-1}\mathds{1}_{C}\left(\bm \xi_\ell,\bm s_\ell\right)\mathds{1}_{A\times B}\left(\bm \xi_j,\bm s_j\right)\middle| ( \bm s_i)_i\right]\\
\stackrel{\eqref{eq:lim_of_indicator}}{=} &\lim_{n\to \infty}\lim_{m\to \infty}
\frac{1}{4n^2}\sum_{\ell,j=1}^{2n-1}\E\left[\mathds{1}_{C_m}\left(\bm \xi_\ell,\bm s_\ell\right)\mathds{1}_{A\times B}\left(\bm \xi_j,\bm s_j\right)\middle|  \bm s_\ell,\bm s_j\right]\\
= &\lim_{m\to \infty}\lim_{n\to \infty}
\frac{1}{4n^2}\sum_{\ell,j=1}^{2n-1}\E\left[\mathds{1}_{C_m}\left(\bm \xi_\ell,\bm s_\ell\right)\mathds{1}_{A\times B}\left(\bm \xi_j,\bm s_j\right)\middle|  \bm s_\ell,\bm s_j\right]\\
=
&\lim_{m\to \infty}\E\left[\mathds{1}_{C_m}(\bm V,\bm U)\right]\E\left[\mathds{1}_{A \times B}(\bm V,\bm U)\right]=\E\left[\mathds{1}_{C}(\bm V,\bm U)\right]\E\left[\mathds{1}_{A \times B}(\bm V,\bm U)\right],
\end{align}
where the last equality follows by monotone convergence. This proves \cref{eq:ifbueiwbfowebnfoewinf} and concludes the proof (up to proving item (b)) that $\mathcal{A}_{A\times B}$ is a monotone class.

\textbf{Proof of  item (b). }Since $\mathds{1}_{C\setminus C_m}=\mathds{1}_{C}-\mathds{1}_{C_m}$, it is enough to show that
\begin{equation}
\sup_{n}\E\left[\frac{1}{4n^2}\sum_{\ell,j=1}^{2n-1}\mathds{1}_{C\setminus C_m}\left(\bm \xi_\ell,\bm s_\ell\right)\mathds{1}_{A\times B}\left(\bm \xi_j,\bm s_j\right)\middle| ( \bm s_i)_i\right]\xrightarrow[m\to\infty]{a.s.}0.
\end{equation}
We set $D_m\coloneqq C\setminus C_m$, and define 
$$(D_m)^{s}\coloneqq\{x\in \R_+|(x,s)\in D_m\},\quad \text{for all} \quad s\in\R,$$
 
$$\pi_Y(D_m)\coloneqq\{y\in \R_+| \exists x\in \R_+ \text{ s.t. }(x,y)\in D_m\}.$$ 

Since  $\bm \xi_\ell\stackrel{d}{=}\bm V$ for all $\ell\in[2n-1]$, then for all $\ell,j\in[2n-1]$, a.s.
\begin{align}
\E\left[\mathds{1}_{D_m}\left(\bm \xi_\ell,\bm s_\ell\right)\mathds{1}_{A\times B}\left(\bm \xi_j,\bm s_j\right)\middle| \bm s_\ell,\bm s_j\right]
&=\P\left(\bm \xi_\ell\in (D_m)^{\bm s_\ell},\bm \xi_j\in A\middle| \bm s_\ell\right)\mathds{1}_{\pi_Y(D_m)}(\bm s_\ell)\mathds{1}_{B}(\bm s_j)\\
&\leq \P\left(\bm \xi_\ell\in (D_m)^{\bm s_\ell}\middle| \bm s_\ell\right)=\P\left(\bm V\in (D_m)^{\bm s_\ell}\middle| \bm s_\ell\right).\label{eq:bound_for_expect1}
\end{align}
Therefore a.s.
\begin{equation}\label{eq:bnd_for_exp}
\sup_{n}\E\left[\frac{1}{4n^2}\sum_{\ell,j=1}^{2n-1}\mathds{1}_{C\setminus C_m}\left(\bm \xi_\ell,\bm s_\ell\right)\mathds{1}_{A\times B}\left(\bm \xi_j,\bm s_j\right)\middle| ( \bm s_i)_i\right]
\leq\P\left(\bm V\in (C\setminus C_m)^{\bm s_\ell}\middle| \bm s_\ell\right).
\end{equation}
Now the sequence $\P\left(\bm V\in (C\setminus C_m)^{\bm s_\ell}\middle| \bm s_\ell\right)$ is a.s.\ non-increasing (because  $C_m \subseteq C_{m+1}$) and hence has an a.s.\ limit. The limit is non-negative and its expectation is the limit of expectations which is $0$ because  $C\coloneqq\bigcup_m C_m$. This completes the proof of item (b).
\end{itemize}

It remains to prove that $\mathcal{A}_{C*}$ is also a monotone class.
The proof is similar to the proof above, replacing the bound in \cref{eq:bound_for_expect1} by
\begin{multline}\label{eq:bound_for_expect2}
\E\left[\mathds{1}_{C^*}\left(\bm \xi_\ell,\bm s_\ell\right)\mathds{1}_{D_m}\left(\bm \xi_j,\bm s_j\right)\middle| \bm s_\ell,\bm s_j\right]\\
=\P\left(\bm \xi_\ell\in (C^*)^{\bm s_\ell},\bm \xi_j\in (D_m)^{\bm s_j}\middle| \bm s_\ell,\bm s_j\right)\mathds{1}_{\pi_Y(C^*)}(\bm s_\ell)\mathds{1}_{\pi_Y(D_m)}(\bm s_j)
\leq \P\left(\bm \xi_j\in (D_m)^{\bm s_j}\middle|\bm s_j\right).
\end{multline}
This completes the proof of the lemma.
\end{proof}

We now generalize the result in \cref{lem:ind_fnct} to all bounded and measurable functions, hereby proving \cref{prop:conv_for_bnd_cont_funct}.

\begin{proof}[Proof of \cref{prop:conv_for_bnd_cont_funct}]
We further assume that $g:\R^2_+\to\R_+$ is non-negative, the general case following by stardand arguments.
Fubini's theorem, together with the fact that $g(x,y)=\int_0^{\|g\|_\infty}\mathds{1}_{\left\{z\leq g(x,y)\right\}}dz$, yields
\begin{multline}
\E\left[\left(\frac{1}{2n}\sum_{j=1}^{2n-1}g\left(\bm \xi_j,s_j\right)\right)^2\middle| (\bm s_i)_i\right]
=\E\left[\frac{1}{4n^2}\sum_{\ell,j=1}^{2n-1}\int_0^{\|g\|_\infty}\mathds{1}_{\left\{z\leq g\left(\bm \xi_\ell,s_\ell\right)\right\}}dz \int_0^{\|g\|_\infty}\mathds{1}_{\left\{t\leq g\left(\bm \xi_j,s_j\right)\right\}}dt\middle| (\bm s_i)_i\right]\\
=\int_0^{\|g\|_\infty}\int_0^{\|g\|_\infty}\frac{1}{4n^2}\sum_{\ell,j=1}^{2n-1}\E\left[\mathds{1}_{A(z)}\left(\bm \xi_\ell,\bm s_\ell\right)\mathds{1}_{A(t)}\left(\bm \xi_j,\bm s_j\right)\middle| \bm s_\ell,\bm s_j\right]dz \; dt,
\end{multline}
where  $A(s)=\left\{(x,y)\in\R_+^2 \middle | g(x,y)\geq s \right\}$. 
%As before, for almost every $(z,t)\in\R^2_+$, $\nu \times \mu \left(\partial A(z)\right)=0=\nu \times \mu \left(\partial A(t)\right)$. 
By \cref{lem:ind_fnct}, for almost every $(z,t)\in\R^2_+$
\begin{equation}
\frac{1}{4n^2}\sum_{\ell,j=1}^{2n-1}\E\left[\mathds{1}_{A(z)}\left(\bm \xi_\ell,\bm s_\ell\right)\mathds{1}_{A(t)}\left(\bm \xi_j,\bm s_j\right)\middle| (\bm s_i)_i\right]\xrightarrow[n\to \infty]{a.s.} \E\left[\mathds{1}_{A(z)}\left(\bm V,\bm U\right)\right]\E\left[\mathds{1}_{A(t)}\left(\bm V,\bm U\right)\right].
\end{equation}
Since the left-hand side is bounded by $1$, we can conclude by dominated convergence that
\begin{equation}
\E\left[\left(\frac{1}{2n}\sum_{j=1}^{2n-1}g\left(\bm \xi_j,s_j\right)\right)^2\middle| (\bm s_i)_i\right]\xrightarrow[n\to \infty]{a.s}\E\left[g\left(\bm V,\bm U\right)\right]^2,
\end{equation} 
completing the proof of the proposition.
\end{proof}

\section{Proof of the central limit theorems}\label{sect:CLT}

In this section, we start by proving \cref{thm:CLT} using \cref{prop:clt_sum_of_the_g} and then we prove the latter result.
\begin{proof}[Proof of \cref{thm:CLT}]
	We set $\bm X_n\coloneqq\frac{\bm W_n(s)- 2n \cdot \E_{\vec{s}}[\bm W_n(s)]}{\sqrt{2n}}$. In order to prove the convergence in \cref{eq:CLT}, it is enough to show that, for every $t \in \mathbb{R}$,
	\begin{equation}\label{eq:goal_proof_MGF}
	\E_{\vec{s}} \left[e^{it\bm X_n}\right]\xrightarrow{n\to \infty}e^{-\frac{t^2}{2} \left(\sigma(s)^2+\rho(s)^2 \right)},
	\end{equation}
	where we recall that $\sigma(s)^2=\E\left[F_s(\bm V,\bm U)-F^2_s(\bm V,\bm U)\right]$ and $\rho(s)^2 = \rho_{F_s}^2$. Note that $\sigma(s)^2+\rho(s)^2$ is finite thanks to \cref{prop:clt_sum_of_the_g} and the fact that $F_s-F_s^2$ is a bounded and measurable function.

	Recalling that $\bm W_n(s)=\sum_{j=1}^{2n-1}\mathds{1}_{W_j}$, $W_{j}$ being the event that the team $T_{0}$ wins against the team $T_j$, 
	and that, conditioning on $\left( \bm \xi^{0}_r \right)_{_{r \in [2n-1]}}$, the results of different matches are independent, we have that
	\begin{multline}
	\E_{\vec{s}}\left[e^{it\bm X_n}\right]
	=e^{-\sqrt{2n} \cdot  \E_{\vec{s}}[\bm W_n(s)] \cdot it} \cdot \E_{\vec{s}} \left[e^{\frac{it}{\sqrt{2n}} \sum_{j=1}^{2n-1}\mathds{1}_{W_j}}\right]\\
	=e^{-\sqrt{2n}  \cdot  \E_{\vec{s}}[\bm W_n(s)] \cdot it}\cdot \E_{\vec{s}}\left[\E_{\vec{s}}\left[e^{\frac{it}{\sqrt{2n}}\sum_{j=1}^{2n-1}\mathds{1}_{W_j}}\middle|\left( \bm \xi^{0}_r \right)_{_{r \in [2n-1]}}  \right]\right] \\
	=e^{-\sqrt{2n} \cdot  \E_{\vec{s}}[\bm W_n(s)] \cdot it}\cdot \E_{\vec{s}}\left[ \prod_{j=1}^{2n-1} \E_{\vec{s}}\left[e^{\frac{it}{\sqrt{2n}}\mathds{1}_{W_j}}\middle| \bm \xi^{0}_j   \right]\right].
	\end{multline}
	Since, by assumption, we have that
	$
	\P_{\vec{s}}\left(W_j\middle | \bm \xi^{0}_j,\bm \xi^{j}_j\right)=f(s\cdot \bm \xi^{0}_j,s_j \cdot \bm \xi^{j}_j)
	$ 
	and, for all $j\in [2n-1]$, $\bm \xi^j_j$ is independent of $\bm \xi^{0}_j$, $\bm \xi^{0}_j\stackrel{d}{=}\bm \xi_j$ and $\bm \xi^j_j\stackrel{d}{=}\bm V'$, we have that
	\begin{multline}
	\E_{\vec{s}}\left[e^{it\bm X_n}\right] 
	= e^{-\sqrt{2n} \cdot  \E_{\vec{s}}[\bm W_n(s)] \cdot it}\cdot \E\left[ \prod_{j=1}^{2n-1} \E\left[ 1 + \left( e^{\frac{it}{\sqrt{2n}}} -1 \right) f(s\cdot \bm \xi^{0}_j,s_j \cdot \bm \xi^{j}_j) \middle| \bm \xi^{0}_j \right]\right] \\
	= e^{-\sqrt{2n} \cdot  \E_{\vec{s}}[\bm W_n(s)] \cdot it}\cdot \E\left[ \prod_{j=1}^{2n-1} \left(1 + \left( e^{\frac{it}{\sqrt{2n}}} -1 \right) \cdot F_s(\bm \xi_j,s_j) \right)\right],
	\end{multline}
	where we recall that $F_s(x,y)=\E\left[f\left(s\cdot x,y\cdot \bm V'\right)\right]$.
	Rewriting the last term as
	\begin{equation}
	e^{-\sqrt{2n}  \cdot  \E_{\vec{s}}[\bm W_n(s)] \cdot  it}\cdot \E\left[ e^{ \sum_{j=1}^{2n-1} \log  \left(1 +  (e^{it /\sqrt{2n}} -1 ) \cdot F_s(\bm \xi_j,s_j ) \right) } \right],
	\end{equation}
	and observing that
	\begin{multline}
	\sum_{j=1}^{2n-1}  \log \left(  1 + \left( e^{it /\sqrt{2n}} -1 \right) \cdot F_s\left(\bm \xi_j,s_j \right)  \right) \\
	=  \sum_{j=1}^{2n-1} \left(\frac{it}{\sqrt{2n}}  \cdot F_s\left(\bm \xi_j,s_j \right) - \frac{t^2}{4n} \left( F_s\left(\bm \xi_j,s_j \right) - F_s^2\left(\bm \xi_j,s_j \right) \right) + O \left(\frac{1}{n\sqrt{n}} \right)\right) \\
	=  \frac{i t}{\sqrt{2n}} \sum_{j=1}^{2n-1} F_s\left(\bm \xi_j,s_j \right) - \frac{t^2}{2} \cdot  \frac{1}{2n} \sum_{j=1}^{2n-1}  \left( F_s\left(\bm \xi_j,s_j \right) - F_s^2\left(\bm \xi_j,s_j \right) \right)  + O \left(\frac{1}{\sqrt{n}} \right),
	\end{multline}
	we obtain that the characteristic function  $\E_{\vec{s}}\left[e^{it\bm X_n}\right]$ is equal to
	\begin{equation}
	e^{O\left( \frac{1}{\sqrt n} \right)} \cdot e^{-\frac{t^2}{2}\sigma(s)^2}
	\cdot \E \left[	e^{  \frac{it}{\sqrt{2n}} \left( \sum_{j=1}^{2n-1} F_s (\bm \xi_j,s_j ) - 2n \cdot  \E_{\vec{s}}[\bm W_n(s)] \right) }
	\cdot e^{-  \frac{t^2}{2} \cdot   \left(\frac{1}{2n} \sum_{j=1}^{2n-1}  \left( F_s\left(\bm \xi_j,s_j \right) - F_s^2\left(\bm \xi_j,s_j \right) \right) - \sigma(s)^2  \right) } \right].
	\end{equation}	
	Now we set
	\begin{align}
		&\bm A_n\coloneqq e^{  \frac{it}{\sqrt{2n}} \left( \sum_{j=1}^{2n-1} F_s (\bm \xi_j,s_j ) - 2n \cdot  \E_{\vec{s}}[\bm W_n(s)] \right) },\\
		&\bm B_n\coloneqq e^{-  \frac{t^2}{2} \cdot   \left(\frac{1}{2n} \sum_{j=1}^{2n-1}  \left( F_s\left(\bm \xi_j,s_j \right) - F_s^2\left(\bm \xi_j,s_j \right) \right) - \sigma(s)^2  \right) },
	\end{align}
	obtaining that $\E_{\vec{s}}\left[e^{it\bm X_n}\right]=	e^{O\left( \frac{1}{\sqrt n} \right)} \cdot e^{-\frac{t^2}{2}\sigma(s)^2}\left(\E\left[\bm A_n\right]-\E\left[\bm A_n(1-\bm B_n)\right]\right)$.
	Hence, \cref{eq:goal_proof_MGF} holds if we show that
	\begin{enumerate}
		\item $\E\left[\bm A_n\right] \to e^{-\frac{t^2}{2}  \rho(s)^2} $, \label{eq:clt_delta_small2}
		\item $\E\left[\bm A_n(1-\bm B_n)\right] \to 0 $ .\label{eq:clt_delta_big2}
	\end{enumerate}
	Item 1 follows from \cref{prop:clt_sum_of_the_g}. For Item 2, since $|\bm A_n|=1$, we have that
	\begin{equation}
		\left|\E\left[\bm A_n(1-\bm B_n)\right] \right|\leq \E\left[|1-\bm B_n|\right].
	\end{equation}
	Recalling that $\sigma(s)^2=\E\left[F_s(\bm V,\bm U)-F^2_s(\bm V,\bm U)\right]$, and that $\bm \xi_j\stackrel{d}{=}\bm V$ for all $j\in[2n-1]$, we have that
	\begin{multline}
		\frac{1}{{2n}}  \sum_{j=1}^{2n-1}  \left( F_s\left(\bm \xi_j,s_j \right) - F_s^2\left(\bm \xi_j,s_j \right) \right)  -   \sigma(s)^2
		=\\
		\frac{1}{{2n}}  \sum_{j=1}^{2n-1}  \left( F_s\left(\bm \xi_j,s_j \right) - F_s^2\left(\bm \xi_j,s_j \right) \right)
		-\frac{1}{{2n}}  \sum_{j=1}^{2n-1}  \E\left[ F_s\left(\bm V,s_j \right) - F_s^2\left( \bm V,s_j \right) \right]\\
		+\frac{1}{{2n}}  \sum_{j=1}^{2n-1}  \E\left[ F_s\left(\bm V,s_j \right) - F_s^2\left( \bm V,s_j \right) \right]
		-\E\left[F_s(\bm V,\bm U)-F^2_s(\bm V,\bm U)\right]   \xrightarrow{P} 0,
	\end{multline}
	where for the limit we used once again \cref{prop:clt_sum_of_the_g} and similar arguments to the ones already used in the proof of \cref{prop:first_mom}.
	Since the function $e^{-t^2x/2}$ is continuous and the random variable $\frac{1}{{2n}} \left( \sum_{j=1}^{2n-1}  F_s\left(\bm \xi_j,s_j \right) - F_s^2\left(\bm \xi_j,s_j \right) \right) -   \sigma(s)^2$ is bounded, we can conclude that $\E\left[|1-\bm B_n|\right]\to 0$. This ends the proof of \cref{thm:CLT}.
\end{proof}

The rest of this section is devoted to the proof of \cref{prop:clt_sum_of_the_g}. We start by stating a lemma that shows how the coefficients $\alpha_n$ defined in \cref{eq:def_alpha_n} control the correlations of the process $\bm \xi$. 
\begin{lem}[Theorem 17.2.1 in \cite{MR0322926}]\label{lem:decay_correlations}
	Fix $\tau \in \mathbb{N}$ and let $\bm X$ be a random variable measurable w.r.t.\ $\mathcal{A}_1^{k}$ and $\bm Y$ a random variable measurable w.r.t.\ $\mathcal{A}^{\infty}_{k + \tau}$. Assume, in addition, that $| \bm X | < C_1$ almost surely and $| \bm Y | < C_2$ almost surely. Then 
	\begin{equation}\label{eq:decay_correlations}
	\left|	\E \left[\bm X \bm Y\right] - \E [\bm X] \E[\bm Y]  \right| \leq 4 \cdot C_1 \cdot C_2 \cdot \alpha_{\tau}.
	\end{equation}
\end{lem}

We now focus on the behaviour of the random variables $\sum_{j=1}^{2n-1}g\left(\bm \xi_j,s_j\right)$ appearing in the statement of \cref{prop:clt_sum_of_the_g}. It follows directly from \cref{prop:conv_for_bnd_cont_funct} and 
Chebyshev's inequality that, for $\mu^{\N}$-almost every sequence $(s_i)_{{i\in\N}}\in\R^{\N}_{+}$,
\begin{equation}
	\frac{1}{2n} \sum_{j=1}^{2n-1}g\left(\bm \xi_j,s_j\right) \xrightarrow{P} \E \left[g \left( \bm V, \bm U \right)\right].
\end{equation} 
We aim at establishing a central limit theorem. 
Recalling the definition of the function $\tilde g$ in the statement of \cref{prop:clt_sum_of_the_g}, we note that, for all $j \in \mathbb{N}$,
\begin{equation}
\tilde{g} \left(\bm \xi_j,s_j \right) = g\left(\bm \xi_j,s_j \right) - \E[g\left(\bm V, s_j\right)],
\end{equation}
and so $\E \left[\tilde{g}  \left(\bm \xi_j,s_j \right) \right] = 0$.
Define
\begin{equation}
	\rho_{g,n}^2 \coloneqq \Var \left( \sum_{j=1}^{2n-1}\tilde{g} \left(\bm \xi_j,s_j\right) \right) =  \Var \left( \sum_{j=1}^{2n-1}g\left(\bm \xi_j,s_j\right) \right).
\end{equation}
The following lemma shows that the variance $\rho_{g,n}^2$ is asymptotically linear in $n$ and proves the first part of \cref{prop:clt_sum_of_the_g}.

\begin{lem} \label{lem:variance_is_linear}
The quantity $\rho_g^2$ defined in \cref{eq:def_of_rho} is finite. Moreover, for $\mu^{\N}$-almost every sequence $(s_i)_{{i\in\N}}\in\R^{\N}_{+}$, we have that 
	\begin{equation} \label{eq:asympt_for_var}
\rho_{g,n}^2 = 2n\cdot \rho_g^2 \cdot (1+o(1)).
\end{equation}
\end{lem}
\begin{proof}
	We have that
	\begin{multline}
	\rho_{g,n}^2 =	\Var \left( \sum_{j=1}^{2n-1}\tilde{g} \left(\bm \xi_j,s_j\right) \right) 
		= \E \left[ \left( \sum_{j=1}^{2n-1}\tilde{g} \left(\bm \xi_j,s_j\right) \right)^2  \right] \\
		 = \E \left[  \sum_{j=1}^{2n-1}\tilde{g}^2 \left(\bm \xi_j,s_j\right)  \right] 
		 +2\cdot  \E \left[ \sum_{i=1}^{2n-2} \sum_{j=i+1}^{2n-1}  \tilde{g} \left(\bm \xi_i,s_i\right) \tilde{g} \left(\bm \xi_j,s_j\right)  \right] \\
		  = \sum_{j=1}^{2n-1}  \E \left[ \tilde{g}^2 \left(\bm \xi_j,s_j\right)  \right]  
		 + 2\cdot \sum_{i=1}^{2n-2} \sum_{k=1}^{2n-1 -i} \E \left[  \tilde{g} \left(\bm \xi_i,s_i\right) \tilde{g} \left(\bm \xi_{i+k},s_{i+k}\right)  \right].
	\end{multline}
	Using similar arguments to the ones already used in the proof of \cref{prop:first_mom}, we have that 
 	\begin{equation}\label{eq:fibiwfwofnbew}
		\sum_{j=1}^{2n-1}  \E \left[\tilde{g}^2 \left(\bm \xi_j,s_j\right)  \right]  = 2n \cdot  \E \left[ \tilde{g}^2  (\bm V, \bm U) \right] + o(n) .
	\end{equation}
	We now show that 
	\begin{equation}\label{eq:wfejbbfweqibfdwequobfd}
	\lim_{n\to \infty} \frac{1}{2n} \sum_{i=1}^{2n-2} \sum_{k=1}^{2n-1 -i} \E \left[  \tilde{g} \left(\bm \xi_i,s_i\right) \tilde{g} \left(\bm \xi_{i+k},s_{i+k}\right)  \right]  =	\sum_{k=1}^{\infty} \E\left[ \tilde{g} (\bm{\xi}_1, \bm U) \tilde{g}  (\bm{\xi}_{1+k}, \bm U') \right].
	\end{equation}
	First we show that the right-hand side is convergent. We start by noting that from Fubini's theorem and \cref{lem:decay_correlations},
	$$ \E \left[\E \left[ \tilde{g} (\bm{\xi}_1, \bm U) \tilde{g}  (\bm{\xi}_{1+k}, \bm U') \middle| \bm{\xi}_1, \bm{\xi}_{1+k} \right]\right] =\int_{\R^2} \E \left[\tilde{g} (\bm{\xi}_1, x) \tilde{g}  (\bm{\xi}_{1+k}, y)\right]  d\mu(x) d\mu(y)\leq 4\cdot \alpha_k.$$
	Therefore, thanks to  the assumption in \cref{eq:strongly_mix_plus}, we have that
	\begin{multline}\label{eq:bnd_with_alpha}
		\sum_{k=1}^{\infty} \E\left[ \tilde{g} (\bm{\xi}_1, \bm U) \tilde{g}  (\bm{\xi}_{1+k}, \bm U') \right] =	\sum_{k=1}^{\infty} \E\left[ \E \left[ \tilde{g} (\bm{\xi}_1, \bm U) \tilde{g}  (\bm{\xi}_{1+k}, \bm U') \middle| \bm{\xi}_1, \bm{\xi}_{1+k} \right] \right] 
		\leq \sum_{k=1}^{\infty} 4 \cdot \alpha_k < \infty.
	\end{multline}  
	
	Now we turn to the proof of the limit in \cref{eq:wfejbbfweqibfdwequobfd}. Using the stationarity assumption for the process $\bm\xi$ in \cref{eq:stationarity}, we can write
	\begin{equation}
		\frac{1}{2n} \sum_{i=1}^{2n-2} \sum_{k=1}^{2n-1 -i} \E \left[ \tilde{g} \left(\bm \xi_i,s_i\right) \tilde{g} \left(\bm \xi_{i+k},s_{i+k}\right)  \right] 
		=  \sum_{k=1}^{2n-2}  	\frac{1}{2n} \sum_{i=1}^{2n-1-k}  \E \left[  \tilde{g} \left(\bm \xi_1,s_i\right) \tilde{g} \left(\bm \xi_{1+k},s_{i+k}\right)  \right].
	\end{equation}
	Using a monotone class argument similar to the one used for the law of large numbers, we will show that the right-hand side of the equation above converges to 
	\begin{equation}
	 \sum_{k=1}^{\infty} \E\left[ \tilde{g} (\bm{\xi}_1, \bm U) \tilde{g}  (\bm{\xi}_{1+k}, \bm U') \right].
	\end{equation}
	We start, as usual, from indicator functions. We have to prove that for all quadruplets $(A,A',B,B')$ of Borel subsets of $\R_+$, it holds that
	\begin{multline}\label{eq:dim_for_ind_fct_centered}
 	\sum_{k=1}^{2n-2}  	\frac{1}{2n} \sum_{i=1}^{2n-1-k}  \E \left[ \tilde{\mathds{1}}_{A \times B}\left(\bm \xi_1,s_i\right)  
 	\tilde{\mathds{1}}_{A' \times B'}\left(\bm \xi_{1+k},s_{i+k}\right) \right] \\
	 \to 
	 \sum_{k=1}^{\infty} \E\left[ \tilde {\mathds{1}}_{A \times B}(\bm{\xi}_1, \bm U)   \tilde {\mathds{1}}_{A' \times B'} (\bm{\xi}_{1+k}, \bm U') \right] <\infty,
\end{multline}
	where for every rectangle $R$ of $\R_+^2$,
\begin{align}
	\tilde{\mathds{1}}_{R}\left(x,y\right)\coloneqq\mathds{1}_{R}\left(x,y\right)-\E\left[ \mathds{1}_{R}\left(\bm \xi_1,y\right) \right].
		\end{align}
	Setting $S_n\coloneqq\sum_{k=1}^{n}\E\left[ \tilde {\mathds{1}}_{A \times B}(\bm{\xi}_1, \bm U)   \tilde {\mathds{1}}_{A' \times B'} (\bm{\xi}_{1+k}, \bm U') \right]$, we estimate
		\begin{multline}\label{eq:rehbgre0uq-9grg}
	\left|	S_{\infty} -  
	\sum_{k=1}^{2n-2}  	\frac{1}{2n} \sum_{i=1}^{2n-1-k}  \E \left[ \tilde{\mathds{1}}_{A \times B}\left(\bm \xi_1,s_i\right) \tilde{\mathds{1}}_{A' \times B'}\left(\bm \xi_{1+k},s_{i+k}\right)  \right] 	\right| \\
	\leq \left| S_{\infty}-S_{2n-2}\right| 
	+  \left|	S_{2n-2} - 
	\sum_{k=1}^{2n-2}  	\frac{1}{2n} \sum_{i=1}^{2n-2}  \E \left[ \tilde{\mathds{1}}_{A \times B}\left(\bm \xi_1,s_i\right) \tilde{\mathds{1}}_{A' \times B'}\left(\bm \xi_{1+k},s_{i+k}\right)  \right] 	\right| \\
	+  \left|	\sum_{k=1}^{2n-2}  	\frac{1}{2n} \sum_{i=2n-1-k}^{2n-2}  \E \left[ \tilde{\mathds{1}}_{A \times B}\left(\bm \xi_1,s_i\right) \tilde{\mathds{1}}_{A' \times B'}\left(\bm \xi_{1+k},s_{i+k}\right)  \right] \right|.
	\end{multline}
	Clearly, the first term in the right-hand side of the inequality above tends to zero, being the tail of a convergent series (the fact that $S_{\infty}<\infty$ follows via arguments already used for \cref{eq:bnd_with_alpha}). For the last term, we notice that, using \cref{lem:decay_correlations},
	\begin{equation}
		\left| \E \left[ \tilde{\mathds{1}}_{A \times B}\left(\bm \xi_1,s_i\right) \tilde{\mathds{1}}_{A' \times B'}\left(\bm \xi_{1+k},s_{i+k}\right)  \right] \right| \leq 4 \cdot \alpha_k,
	\end{equation}
	and thus
	\begin{equation}
	 \left|	\sum_{k=1}^{2n-2}  	\frac{1}{2n} \sum_{i=2n-1-k}^{2n-2}  \E \left[ \tilde{\mathds{1}}_{A \times B}\left(\bm \xi_1,s_i\right) \tilde{\mathds{1}}_{A' \times B'}\left(\bm \xi_{1+k},s_{i+k}\right)  \right] \right| \leq \frac{1}{2n}	\sum_{k=1}^{2n-2}  4 k \cdot \alpha_k,
	\end{equation}
	which converges to $0$ as $n$ goes to infinity by the assumption in \cref{eq:strongly_mix_plus} and the same arguments used in \cref{rem:fbkwufobw}.
	It remains to bound the second term. Expanding the products and recalling that $\bm V, \bm V', \bm U, \bm U'$ are independent random variables such that $\bm{\xi}_{1}\stackrel{d}{=}\bm{\xi}_{1+k}\stackrel{d}{=}\bm {V}\stackrel{d}{=}\bm {V}'$ and $\bm {U}\stackrel{d}{=}\bm {U}'\stackrel{d}{=}\mu$, we have that
	\begin{multline}
	S_{2n-2}=\sum_{k=1}^{2n-2} \E\left[ \tilde {\mathds{1}}_{A \times B}(\bm{\xi}_1, \bm U)   \tilde {\mathds{1}}_{A' \times B'} (\bm{\xi}_{1+k}, \bm U') \right]\\
	=
	\sum_{k=1}^{2n-2} \E\left[  \mathds{1}_{A \times B}(\bm{\xi}_1, \bm U)  \mathds{1}_{A' \times B'} (\bm{\xi}_{1+k}, \bm U') \right]-\E\left[  \mathds{1}_{A \times B}(\bm V, \bm U)  \mathds{1}_{A' \times B'} (\bm{V}', \bm U') \right]\\
	=
	\sum_{k=1}^{2n-2} \mu \left(B\right) \cdot \mu \left(B'\right) \cdot
	\left(\E\left[  \mathds{1}_{A \times A'}(\bm{\xi}_1,\bm{\xi}_{1+k} )\right]-\E\left[  \mathds{1}_{A \times A'}(\bm V,\bm{V}') \right]\right).
	\end{multline}
	Similarly, we obtain 
	\begin{multline}
	\E \left[ \tilde{\mathds{1}}_{A \times B}\left(\bm \xi_1,s_i\right)  
	\tilde{\mathds{1}}_{A' \times B'}\left(\bm \xi_{1+k},s_{i+k}\right) \right]\\
	=
	\E \left[ \mathds{1}_{A \times B}\left(\bm \xi_1,s_i\right) \mathds{1}_{A' \times B'}\left(\bm \xi_{1+k},s_{i+k}\right)\right]-\E \left[ \mathds{1}_{A \times B}\left( \bm V,s_i\right)\right]\E\left[ \mathds{1}_{A' \times B'}\left(\bm V',s_{i+k}\right)\right]\\
	= \mathds{1}_{B \times B'}(s_i,s_{i+k})\cdot \left(\E \left[ \mathds{1}_{A \times A'}\left(\bm \xi_1,\bm \xi_{1+k}\right)\right]-\E \left[ \mathds{1}_{A \times A'}\left( \bm V,\bm V'\right)\right]\right).
	\end{multline}
	Therefore the second term in the right-hand side of \cref{eq:rehbgre0uq-9grg} is bounded by 
	\begin{equation}\label{eq:erbgorobgegoe}
	\sum_{k=1}^{2n-2}   \left|   \E \left[ \mathds{1}_{A \times A'}\left(\bm \xi_1,\bm \xi_{1+k}\right)\right]-\E \left[ \mathds{1}_{A \times A'}\left( \bm V,\bm V'\right)\right]\right| \cdot \left|  \mu (B)\mu(B') 	-\frac{1}{2n} \sum_{i=1}^{2n-2} \mathds{1}_{B \times B'} (s_i, s_{i+k})  \right|.
	\end{equation}
	Using \cref{prop:uniform_bound} we have that
	$$\sup_{k \in [2n-2]} \left| \mu (B)\mu(B') 	-\frac{1}{2n} \sum_{i=1}^{2n-2} \mathds{1}_{B \times B'} (\bm s_i, \bm s_{i+k})  \right| \xrightarrow[n\to\infty]{a.s.} 0.$$
	In addition, using once again \cref{lem:decay_correlations} and the assumption in \cref{eq:strongly_mix_plus} we have that  $$\sum_{k=1}^{2n-2}   \left|   \E \left[ \mathds{1}_{A \times A'}\left(\bm \xi_1,\bm \xi_{1+k}\right)\right]-\E \left[ \mathds{1}_{A \times A'}\left( \bm V,\bm V'\right)\right]\right|<\infty.$$
	The last two equations imply that the bound in \cref{eq:erbgorobgegoe} tends to zero as $n$ tends to infinity for $\mu^{\N}$-almost every sequence $(s_i)_{{i\in\N}}\in\R^{\N}_{+}$, completing the proof of \cref{eq:dim_for_ind_fct_centered}.
	
	In order to conclude the proof of the lemma, it remains to generalize the result in \cref{eq:dim_for_ind_fct_centered} to all bounded and measurable functions. This can be done using the same techniques adopted to prove the law of large numbers, therefore we skip the details.
\end{proof}

We now complete the proof of \cref{prop:clt_sum_of_the_g}.

\begin{proof}[Proof of \cref{prop:clt_sum_of_the_g}]
	Recalling that $\tilde g (x,y) \coloneqq g(x,y) - \E \left[ g\left( \bm V, y\right) \right]$ and thanks to  \cref{lem:variance_is_linear}, it is enough to show that 
	\begin{equation}\label{eq:clt_h}
		\frac{1}{\rho_{g,n}}	\sum_{j=1}^{2n-1}\tilde{g}\left(\bm \xi_j,s_j\right)\xrightarrow{d} \bm{\mathcal{N}}(0, 1).
	\end{equation}
	The difficulty in establishing this convergence lies in the fact that we are dealing with a sum of random variables that are neither independent nor identically distributed. 
	We proceed in two steps. First, we apply the the Bernstein's method, thus we reduce the problem to the study of a sum of ``almost" independent random variables. More precisely, we use the decay of the correlations for the process $\bm \xi$ to decompose $\sum_{j=1}^{2n-1}\tilde{g}\left(\bm \xi_j,s_j\right)$ into two distinct sums, in such a way that one of them is a sum of ``almost" independent random variables and the other one is negligible (in a sense that will be specified in due time). After having dealt with the lack of independence, we settle the issue that the random variables are not identically distributed using the Lyapounov's condition.
	
	We start with the first step. Recall that we assume the existence of two sequences $p=p(n)$ and $q=q(n)$ such that:
	\begin{itemize}
		\item  	$p\xrightarrow{n\to\infty} +\infty$ and $q\xrightarrow{n\to\infty} +\infty$,
		\item 	$q=o(p)$ and $p=o(n)$ as $n \to \infty$,
		\item   $n p^{-1 } \alpha_q=o(1)$,
		\item   $  \frac{p}{n}  \cdot \sum_{j=1}^p j \alpha_j = o(1)$.
	\end{itemize}
	As said above, we represent the sum $\sum_{j=1}^{2n-1}\tilde{g}\left(\bm \xi_j,s_j\right)$ as a sum of nearly independent random variables (the \emph{big blocks} of size $p$, denoted $\bm \beta_i$ below) alternating with other terms (the \emph{small blocks} of size $q$, denoted $\bm \gamma_i$ below) whose sum is negligible.
	We define $k = \lfloor (2n-1)/(p+q) \rfloor$.
	We can thus write 
	\begin{equation}
		\sum_{j=1}^{2n-1}\tilde{g}\left(\bm \xi_j,s_j\right) = \sum_{i=0}^{k-1} \bm \beta_i +  \sum_{i=0}^k \bm \gamma_i, 
	\end{equation}
	where, for $0\leq i \leq k-1$,
	\begin{equation}
	 \bm \beta_i= \bm \beta_i (\tilde g, n) \coloneqq \sum_{j=ip+iq+1}^{(i+1)p+iq} \tilde g \left(\bm \xi_j,s_j\right), \quad\quad 
	 \bm \gamma_i= \bm \gamma_i (\tilde g, n) \coloneqq \sum_{j=(i+1)p+iq+1}^{(i+1)p+(i+1)q} \tilde g\left(\bm \xi_j,s_j\right),
	\end{equation}
	and 
	\begin{equation}
	\quad \quad \bm \gamma_k=\bm \gamma_k(\tilde g, n) \coloneqq \sum_{j=kp+kq+1}^{2n-1}  \tilde g\left(\bm \xi_j,s_j\right).
	\end{equation}
	Henceforth, we will omit the dependence on $\tilde g$ and on $n$ simply writing $\bm \beta_i$ and $\bm \gamma_i$, in order to simplify the notation (whenever it is clear). Setting $\bm H_n'\coloneqq \frac{1}{\rho_{g,n}} \sum_{i=0}^{k-1} \bm \beta_i$ and $\bm H_n''\coloneqq \frac{1}{\rho_{g,n}}   \sum_{i=0}^k \bm \gamma_i $, we can write 
	\begin{equation}
	\frac{1}{\rho_{g,n}}\sum_{j=1}^{2n-1}\tilde{g}\left(\bm \xi_j,s_j\right) = 	 \bm H_n' + \bm H_n''.
	\end{equation}
	 The proof of \cref{eq:clt_h} now consists of two steps. First, we show that $\bm H_n''\xrightarrow{P} 0$,  and secondly we show that the characteristic function of $\bm H_n'$ converges to the characteristic function of a standard Gaussian random variable. Then, we can conclude using standard arguments. 

	We start by proving that $\bm H_n''  \xrightarrow{P} 0$. By 
	Chebyshev's inequality and the fact that $\E[\bm H_n'']=0$, it is enough to show that $\E \left[\left( \bm H_n''\right)^2\right] \to 0$ as $n \to \infty$. 
	We can rewrite $\E \left[\left( \bm H_n''\right)^2\right]$ as
	\begin{multline}
	 \frac{1	}{\rho_{g,n}^2} \cdot \E \left[\left(  \sum_{i=0}^{k} \bm \gamma_i \right)^2\right]
		=  \frac{1	}{\rho_{g,n}^2} 
		\left(
		 \E \left[ \sum_{i=0}^{k-1} \bm \gamma_i^2  \right] 
		 +  \E \left[\bm \gamma_k^2 \right]
		+\E \left[  \sum_{ \substack{ i,j=0\\i \neq j} }^{k-1} \bm \gamma_i \bm \gamma_j \right] 
		+ 2 \E \left[ \sum_{i=0}^{k-1} \bm \gamma_i \bm \gamma_k \right] 
		\right).
	\end{multline}
	Note that, by definition of $\bm \gamma_i$ and using \cref{lem:decay_correlations} once again, for $i\neq j$, we have the bounds 
	\begin{equation}\label{eq:ebfgreubfoeqrf}
		\E \left[  \bm \gamma_i \bm \gamma_j \right] \leq q^2 \cdot \alpha_{p(i-j)},
	\end{equation}
	and 
	\begin{equation}
	\E \left[  \bm \gamma_i \bm \gamma_k \right] \leq q \cdot (p+q) \cdot \alpha_{p(k-i)}.
	\end{equation}
	Moreover, by the same argument used in \cref{lem:variance_is_linear}, we have that
		\begin{equation}
	\E \left[  \bm \gamma_i^2  \right] = \rho_g^2 \cdot q \cdot (1+o(1))=O(q)
	\end{equation}
	and 
		\begin{equation}
	\E \left[  \bm \gamma_k^2  \right] = O(p+q) = O(p).
	\end{equation}
	Hence, using \cref{lem:variance_is_linear},
	\begin{equation}
	\frac{1}{\rho_{g,n}^2} \E \left[ \sum_{i=0}^{k-1} \bm \gamma_i^2  \right] = O\left(\frac{kq}{n}\right) = O\left(\frac{q}{p}\right) = o(1)
	\end{equation}
	and 
	\begin{equation}
	 \frac{1	}{\rho_{g,n}^2} \E \left[\bm \gamma_k^2 \right] = o(1).
	\end{equation}
	Using \cref{eq:ebfgreubfoeqrf}, we also have that 
	\begin{multline}
		 \frac{1}{\rho_{g,n}^2} \E \left[  \sum_{i,j=1, i \neq j}^{k-1} \bm \gamma_i \bm \gamma_j \right] 
		 \leq  
		 \frac{2}{\rho_{g,n}^2}  \sum_{j=0}^{k-1} \sum_{i=j+1}^{k-1} q^2 \cdot \alpha_{p(i-j)} 
		 =  \frac{2}{\rho_{g,n}^2}  \sum_{j=0}^{k-1} \sum_{m=1}^{k-j} q^2 \cdot \alpha_{pm} \\
		  \leq \frac{2}{\rho_{g,n}^2}\cdot kq^2  \sum_{m=1}^{k}  \alpha_{pm} \leq \frac{2kq^2}{\rho_{g,n}^2\cdot p} \sum_{m=1}^{\infty}  \alpha_{m}= o(1),
	\end{multline}
	where in the last inequality we used that, since $\alpha_n$ is decreasing, 
	\begin{equation}
	 \sum_{m=1}^{k}  \alpha_{pm} \leq  \sum_{m=1}^{k} \frac{1}{p} \cdot \sum_{s=(m-1)p+1}^{mp} \alpha_s  \leq \frac{1}{p}  \cdot \sum_{s=1}^{\infty} \alpha_s.
	\end{equation}
	Analogously, we can prove that 
	\begin{equation}
	 \frac{2}{\rho_{g,n}^2} \E \left[ \sum_{i=0}^{k-1} \bm \gamma_i \bm \gamma_k \right] = o(1),
	\end{equation}
	concluding the proof that $\bm H_n''  \xrightarrow{P} 0$.
	
	Now we turn to the study of the limiting distribution of $\bm H_n'$. 
	We have that, for $t \in \mathbb{R}$,
	\begin{equation}
	\exp \left\{ it \bm H_n' \right\}  =	\exp \left\{ \frac{it}{\rho_{g,n}} \sum_{i=0}^{k-1} \bm \beta_i  \right\}.
	\end{equation}
	We now look at 
	$\exp \left\{ \frac{it}{\rho_{g,n}} \sum_{i=1}^{k-2} \bm \beta_i  \right\}$ 
	and 
	$ 	\exp \left\{ \frac{it}{\rho_{g,n}}  \bm \beta_{k-1} \right\}$.
	We have that the first random variable is measurable with respect to $\mathcal A_1^{(k-1)p + (k-2)q}$ and the second one is measurable with respect to $\mathcal A_{(k-1)p + (k-1)q+1}^{\infty}$. So, by \cref{lem:decay_correlations}, 
	\begin{equation}
		\left| \E \left[\exp \left\{ \frac{it}{\rho_{g,n}}\sum_{i=0}^{k-1} \bm \beta_i \right\} \right] - \E \left[	\exp \left\{ \frac{it}{\rho_{g,n}}\sum_{i=0}^{k-2} \bm \beta_i \right\}\right] \E \left[ \exp \left\{ \frac{it}{\rho_{g,n}}  \bm \beta_{k-1} \right\}\right]\right| \leq 4 \cdot \alpha_q.
	\end{equation}
	Iterating, we get
		\begin{equation}\label{eq:char_function_ofh_converges}
	\left| \E \left[\exp \left\{ \frac{it}{\rho_{g,n}}\sum_{i=0}^{k-1} \bm \beta_i \right\} \right] - 
	\prod_{i=0}^{k-1} \E \left[ \exp \left\{ \frac{it}{\rho_{g,n}}  \bm \beta_{i} \right\}\right]\right|  \leq 4 \cdot (k-1) \cdot \alpha_q, 
	\end{equation}
	the latter quantity tending to $0$ as $k \to \infty$ thanks to the assumptions on the sequences $p$ and $q$.		
	The last step of the proof consists in showing that, as $n \to \infty$,
	\begin{equation}\label{eq:ifbueiwufbweonfew}
	\prod_{i=0}^{k-1} \E \left[ \exp \left\{ \frac{it}{\rho_{g,n}}  \bm \beta_{i} \right\}\right] \to e^{ \frac{t^2}{2} }.
	\end{equation}
	Consider a collection of independent random variables $\bm X_{n,i}$, $n \in \mathbb{N}, i \in [k-1]_0$, such that $\bm X_{n,i}$ has the same distribution as $\frac{1}{\rho_{g,n}}\bm \beta_i (n)$.   
	By  \cite[Theorem 27.3]{MR1324786}, a sufficient condition to ensure that $\sum_{i=0}^{k-1}\bm X_{n,i}\to \bm{\mathcal N}(0,1)$ and so verifying \cref{eq:ifbueiwufbweonfew},
	is the well-known Lyapounov's condition:
	\begin{equation}\label{eq:fbbfoqehfoiewqhf}
		 \lim_{n \to \infty} \frac{1	}{\rho_{g,n}^{2+\delta}} \sum_{i=0}^{k-1} \E \left[ \bm Y_{n,i}^{2+\delta}    \right] = 0, \quad \text{for some} \quad \delta>0,
	\end{equation}
	where $\bm Y_{n,i} \coloneqq \bm X_{n,i} \cdot \rho_{g,n} \stackrel{d}{=} \sum_{j=ip+iq+1}^{(i+1)p+iq} \tilde{g} \left( \bm \xi_j, s_j \right)$.
	The condition is satisfied with $\delta=2$ thanks to the following lemma. 
	\begin{lem}\label{lem:CLT_bound_fourth_moment}
		Under the assumptions of  \cref{prop:clt_sum_of_the_g}, we have that for $\mu^{\N}$-almost every sequence $(s_i)_{{i\in\N}}\in\R^{\N}_{+}$, uniformly for all $i \in [k-1]_0$, 
		\begin{equation}
			\E \left[ \left( \bm Y_{n,i} \right)^4 \right] = O \left( p^2 \cdot \sum_{j=1}^{p} j\alpha_j \right).
		\end{equation}
	\end{lem}
	Before proving the lemma above above, we explain how it implies the condition in \cref{eq:fbbfoqehfoiewqhf} with $\delta=2$. By \cref{lem:variance_is_linear}, we have that $\rho_{g,n}^2 = 2 n \cdot \rho^2_g \cdot (1+o(1))$, thus
	\begin{equation} 
		  \frac{1}{\rho_{g,n}^4 } \sum_{i=0}^{k-1}  \E \left[ \bm Y_{n,i}^4 \right] 
		 \leq 
		  C \cdot \frac{k\cdot p^2 \cdot \sum_{j=1}^{p} j\alpha_j}{4 n^2 \cdot \rho^4_g} \to 0,
	\end{equation}
	where for the limit we used the fact that $ \frac{k\cdot p}{n} \to 1$ and that, by assumption, $ \frac{ p \cdot \sum_{j=1}^{p} j\alpha_j}{n} \to 0$.
	
	\medskip
	
	We conclude the proof of \cref{prop:clt_sum_of_the_g} by proving \cref{lem:CLT_bound_fourth_moment}. Let $A_0 \coloneqq [p]$ and $A_i \coloneqq [(i+1)p+iq] \setminus [ip+iq]$ for $i \geq 1$. Note that $|A_i|=p$ for all $i\geq 0$. We have that
	\begin{align}
	&\E \left[ \left( \bm Y_{n,i} \right)^4 \right]	= \E \left[ \left( \sum_{j=ip+iq+1}^{(i+1)p+iq} \tilde g \left( \bm \xi_j, s_j \right)  \right)^4 \right] \\
	=&O\Bigg(\sum_{j \in A_i}  \E \left[ \tilde g^4 \left( \bm \xi_j, s_j \right) \right] 
	 +  \sum_{\substack {j,k \in A_i \\ j\neq k}} \E \left[ \tilde g ^2 \left( \bm \xi_j, s_j \right) \tilde g ^2 \left( \bm \xi_k, s_k\right) \right] 
	 +  \sum_{\substack{j,k \in A_i\\ j\neq k}} \E \left[ \tilde g^3 \left( \bm \xi_j, s_j \right) \tilde g \left( \bm \xi_k, s_k\right) \right] \\
	 +  &\sum_{\substack{j,k, l \in A_i \\ j\neq k\neq l}} \E \left[ \tilde g^2 \left( \bm \xi_j, s_j \right) \tilde g \left( \bm \xi_k, s_k\right) \tilde g \left( \bm \xi_l, s_l\right) \right] 
	 +  \sum_{\substack{j,k, l,m \in A_i \\ j\neq k\neq l \neq m}} \E \left[ \tilde g\left( \bm \xi_j, s_j \right) \tilde g \left( \bm \xi_k, s_k\right) \tilde g \left( \bm \xi_l, s_l\right) \tilde g \left( \bm \xi_m, s_m\right) \right]\Bigg).
	\end{align} 
	The fact that $\tilde g$ is bounded and the decay of the correlations will give us some bounds for each of these terms. First of all, since $\tilde g$ is bounded, we have that 
	$
	 \sum_{j \in A_i}  \E \left[ \tilde g^4 \left( \bm \xi_j, s_j \right) \right] = O(p),
	$
	$
		 \sum_{j,k \in A_i, j\neq k} \E \left[ \tilde g^2 \left( \bm \xi_j, s_j \right) \tilde g^2 \left( \bm \xi_k, s_k\right) \right] = O \left(p^2 \right),
	$
	and
	$
	 \sum_{j,k \in A_i, j\neq k} \E \left[ \tilde g^3 \left( \bm \xi_j, s_j \right) \tilde g \left( \bm \xi_k, s_k\right) \right] = O \left(p^2 \right).
	$
	We now look at the fourth addend. 
	We have that 
	\begin{multline}
		 \sum_{\substack{j,k, l \in A_i\\ j\neq k\neq l}} \E \left[ \tilde g^2 \left( \bm \xi_j, s_j \right) \tilde g\left( \bm \xi_k, s_k\right) \tilde g \left( \bm \xi_l, s_l\right) \right] \\
		 =
		 O \left(  
		 \sum_{\substack{j,k, l \in A_i\\ j< k< l}} \E \left[ \tilde g^2 \left( \bm \xi_j, s_j \right) \tilde g\left( \bm \xi_k, s_k\right) \tilde g \left( \bm \xi_l, s_l\right) \right]  
		  \right) 
		  =
		  O \left(
		  \sum_{l \in A_i}  \sum_{k = ip+iq+1 }^{l-1}   \sum_{j=ip+iq+1}^{k}  \alpha_{l-k} 
		  \right) = O(p^2),
	\end{multline}
	since  $\sum_{i=1}^{\infty} \alpha_{i} < + \infty$ by assumption and $|A_i|=p$.
	Finally, we estimate the last addend.
	We have that 
	\begin{multline}
		 \sum_{\substack{j,k, l,m \in A_i\\ j\neq k\neq l \neq m}} \E \left[ \tilde g \left( \bm \xi_j, s_j \right) \tilde g\left( \bm \xi_k, s_k\right) \tilde g \left( \bm \xi_l, s_l\right) \tilde g \left( \bm \xi_m, s_m\right) \right]  \\
		=
		O \left(
		 \sum_{\substack{j,k, l,m \in A_i\\ j < k < l <m}} \E \left[ \tilde g\left( \bm \xi_j, s_j \right) \tilde g \left( \bm \xi_k, s_k\right) \tilde g \left( \bm \xi_l, s_l\right) \tilde g\left( \bm \xi_m, s_m\right) \right]
		\right) 
		= 
		O \left(
		 \sum_{\substack{j,k, l,m \in A_i\\j < k < l <m}} \min \{\alpha_{k-j}, \alpha_{m-l}\}
		\right).
	\end{multline}
	We analyse the last expression. Since the sequence $(\alpha_n)_{n\in\N}$ is decreasing, we see that 
	\begin{multline}
		\sum_{\substack{j,k, l,m \in A_i\\j < k < l <m}} \min \{\alpha_{k-j}, \alpha_{m-l}\}
		= 
		\sum_{j \in A_i} \sum_{x =1}^{p-j} \sum_{\substack{l \in A_i \\ l>j+x}} \sum_{y =1}^{p-l} \min \{\alpha_{x}, \alpha_{y}\} \\
		=
		\sum_{j \in A_i} \sum_{x =1}^{p-j} \sum_{\substack{l \in A_i \\ l>j+x}} \sum_{y =1}^{x} \alpha_{x}
		+
		\sum_{j \in A_i} \sum_{x =1}^{p-j} \sum_{\substack{l \in A_i \\ l>j+x}} \sum_{y =x+1}^{p-l} \alpha_{y} .
	\end{multline}
	Since $|A_i|=p$ we have that 
	\begin{equation}
			\sum_{j \in A_i} \sum_{x =1}^{p-j} \sum_{\substack{l \in A_i \\ l>j+x}} \sum_{y =1}^{x} \alpha_{x}
			\leq p^2 \cdot \sum_{x=1}^p x\alpha_{x},
	\end{equation}
	and that
	\begin{equation}
	\sum_{j \in A_i} \sum_{x =1}^{p-j} \sum_{\substack{l \in A_i \\ l>j+x}} \sum_{y =x+1}^{p-l} \alpha_{y}
		\leq 
		p^2 \cdot \sum_{y=1}^p y \alpha_{y},
	\end{equation}
	from which we conclude that 
	\begin{equation}
		 \sum_{\substack{j,k, l,m \in A_i\\ j\neq k\neq l \neq m}} \E \left[ \tilde g \left( \bm \xi_j, s_j \right) \tilde g\left( \bm \xi_k, s_k\right) \tilde g \left( \bm \xi_l, s_l\right) \tilde g \left( \bm \xi_m, s_m\right) \right] 
		= 
		O \left( p^2 \cdot \sum_{j=1}^p j \alpha_j \right).
	\end{equation}
	This concludes the proof of \cref{lem:CLT_bound_fourth_moment}, and hence of \cref{prop:clt_sum_of_the_g} as well.
\end{proof}

\appendix

\section{A uniform estimate for central limit theorems of weakly-correlated random variables}

\begin{prop}\label{prop:uniform_bound}
	Let $(\bm X_i)_{i\in\N}$ be a sequence of i.i.d.\ real-valued random variables. Let also $f:\R^2\to\R$ be a bounded measurable function. Then
	$$\sup_{k \in [n]} \left|\frac{1}{n} \sum_{i=1}^{n} f (\bm X_i, \bm X_{i+k})-f(\bm X_1,\bm X_2)  \right| \xrightarrow[n\to\infty]{a.s.} 0.$$
\end{prop}

\begin{proof}
	Let $A>0$ be such that $f(x,y)\leq A$ for all $x,y\in\R^2$ and set $\bm S_{n,k}\coloneqq  \sum_{i=1}^{n} f (\bm X_i, \bm X_{i+k}).$ From\footnote{Note that the assumptions of \cite[Proposition 5]{MR4105789} are satisfied: indeed by \cite[Theorem 5]{MR4105789} we have that for all $r\geq 1$ the $r$-th cumulant of $S_{n,k}$, denoted by $K^{(r)}(S_{n,k})$ is bounded by,
	$|K^{(r)}(S_{n,k})| \leq n \cdot 4^{r-1}\cdot r^{r-2} \cdot A^r .$} \cite[Proposition 5]{MR4105789},  we have that for any $x>0$ and $k\in[n]$,
	\begin{equation}
	P \left(| S_{n,k} - \E [ S_{n,k} ]| \geq x \right) \leq 2 \exp\left\{-\frac{x^2}{18 A^2 \cdot n}\right\}.
	\end{equation}
	Using this bound together with a union bound, we get that for any $\varepsilon>0$,
	\begin{equation}
	\P\left(\sup_{k \in [n]} \left|\frac{1}{n} \sum_{i=1}^{n} f (\bm X_i, \bm X_{i+k})-f(\bm X_1,\bm X_2)  \right|\geq \varepsilon\right)\leq 2n  \exp\left\{-\frac{n\cdot \varepsilon^2}{18 A^2}\right\}.
	\end{equation}
	Then the statement of the proposition follows using a standard Borel-Cantelli argument.
\end{proof}

\section*{Acknowledgements}
The authors are very grateful to Itai Benjamini, Jean Bertoin, Mathilde Bouvel and Valentin F\'eray for many interesting suggestions and discussions. 
We also thank Emilio Corso for a careful reading of a preliminary draft of the paper.

\bibliography{mybib}

\begin{thebibliography}{BNKK06}

\bibitem[BBCL15]{MR3346459}
M.~Bena\"{\i}m, I.~Benjamini, J.~Chen, and Y.~Lima.
\newblock A generalized {P}\'{o}lya's urn with graph based interactions.
\newblock {\em Random Structures Algorithms}, 46(4):614--634, 2015.

\bibitem[Ber73]{MR350815}
K.~N. Berk.
\newblock A central limit theorem for {$m$}-dependent random variables with
  unbounded {$m$}.
\newblock {\em Ann. Probability}, 1:352--354, 1973.

\bibitem[Bil95]{MR1324786}
P.~Billingsley.
\newblock {\em Probability and measure}.
\newblock Wiley Series in Probability and Mathematical Statistics. John Wiley
  \& Sons, Inc., New York, third edition, 1995.
\newblock A Wiley-Interscience Publication.

\bibitem[BNH07]{ben2007efficiency}
E.~Ben-Naim and N.~Hengartner.
\newblock Efficiency of competitions.
\newblock {\em Physical Review E}, 76(2):026106, 2007.

\bibitem[BNKK06]{ben2006dynamics}
E.~Ben-Naim, B.~Kahng, and J.~S. Kim.
\newblock Dynamics of multi-player games.
\newblock {\em Journal of Statistical Mechanics: Theory and Experiment},
  2006(07):P07001, 2006.

\bibitem[Bor20]{MR4055194}
J.~Borga.
\newblock Local convergence for permutations and local limits for uniform
  {$\rho$}-avoiding permutations with {$|\rho|=3$}.
\newblock {\em Probab. Theory Related Fields}, 176(1-2):449--531, 2020.

\bibitem[BR89]{MR1048950}
P.~Baldi and Y.~Rinott.
\newblock On normal approximations of distributions in terms of dependency
  graphs.
\newblock {\em Ann. Probab.}, 17(4):1646--1650, 1989.

\bibitem[Bra07]{MR2325294}
R.~C. Bradley.
\newblock {\em Introduction to strong mixing conditions. {V}ol. 1}.
\newblock Kendrick Press, Heber City, UT, 2007.

\bibitem[Eks14]{MR3257385}
M.~Ekstr\"{o}m.
\newblock A general central limit theorem for strong mixing sequences.
\newblock {\em Statist. Probab. Lett.}, 94:236--238, 2014.

\bibitem[FMN20]{MR4105789}
V.~F\'{e}ray, P.-L. M\'{e}liot, and A.~Nikeghbali.
\newblock Graphons, permutons and the {T}homa simplex: three mod-{G}aussian
  moduli spaces.
\newblock {\em Proc. Lond. Math. Soc. (3)}, 121(4):876--926, 2020.

\bibitem[HR48]{MR26771}
W.~Hoeffding and H.~Robbins.
\newblock The central limit theorem for dependent random variables.
\newblock {\em Duke Math. J.}, 15:773--780, 1948.

\bibitem[Ibr62]{MR0148125}
I.~A. Ibragimov.
\newblock Some limit theorems for stationary processes.
\newblock {\em Teor. Verojatnost. i Primenen.}, 7:361--392, 1962.

\bibitem[IL71]{MR0322926}
I.~A. Ibragimov and Y.~V. Linnik.
\newblock {\em Independent and stationary sequences of random variables}.
\newblock Wolters-Noordhoff Publishing, Groningen, 1971.
\newblock With a supplementary chapter by I. A. Ibragimov and V. V. Petrov,
  Translation from the Russian edited by J. F. C. Kingman.

\bibitem[Jan88]{MR920273}
S.~Janson.
\newblock Normal convergence by higher semi-invariants with applications to
  sums of dependent random variables and random graphs.
\newblock {\em Ann. Probab.}, 16(1):305--312, 1988.

\bibitem[Jan04]{MR2068873}
S.~Janson.
\newblock Large deviations for sums of partly dependent random variables.
\newblock {\em Random Structures Algorithms}, 24(3):234--248, 2004.

\bibitem[Pel92]{MR1176496}
M.~Peligrad.
\newblock On the central limit theorem for weakly dependent sequences with a
  decomposed strong mixing coefficient.
\newblock {\em Stochastic Process. Appl.}, 42(2):181--193, 1992.

\bibitem[PL82]{MR681466}
M.~B. Petrovskaya and A.~M. Leontovich.
\newblock The central limit theorem for a sequence of random variables with a
  slowly growing number of dependences.
\newblock {\em Teor. Veroyatnost. i Primenen.}, 27(4):757--766, 1982.

\bibitem[PRW97]{MR1492353}
D.~N. Politis, J.~P. Romano, and M.~Wolf.
\newblock Subsampling for heteroskedastic time series.
\newblock {\em J. Econometrics}, 81(2):281--317, 1997.

\bibitem[Ros56]{MR74711}
M.~Rosenblatt.
\newblock A central limit theorem and a strong mixing condition.
\newblock {\em Proc. Nat. Acad. Sci. U.S.A.}, 42:43--47, 1956.

\bibitem[RW00]{MR1747098}
J.~P. Romano and M.~Wolf.
\newblock A more general central limit theorem for {$m$}-dependent random
  variables with unbounded {$m$}.
\newblock {\em Statist. Probab. Lett.}, 47(2):115--124, 2000.

\end{thebibliography}
\bibliographystyle{alpha}

\end{document}